\documentclass{amsart}

\usepackage{amssymb, amsfonts, amsbsy, latexsym}

\newif\ifpdf
\ifx\pdfoutput\undefined
  \pdffalse
\else
  \pdftrue
\fi

\ifpdf
  \usepackage[pdftex]{graphicx}
  \DeclareGraphicsExtensions{.pdf,.jpg,.png}
\else
  \usepackage{graphicx}
\fi

\newtheorem{theorem}{Theorem}[section]
\newtheorem{lemma}[theorem]{Lemma}
\newtheorem{proposition}[theorem]{Proposition}
\newtheorem{corollary}[theorem]{Corollary}
\newtheorem{algorithm}[theorem]{Algorithm}

\theoremstyle{definition}
\newtheorem{definition}[theorem]{Definition}
\newtheorem{notation}[theorem]{Notation}

\theoremstyle{remark}
\newtheorem{remark}[theorem]{Remark}

\def\supp{{\text{\rm supp}}}
\def\span{{\text{\rm span}}}

\def\cP{\mathcal{P}}

\def\N{\mathbb{N}}
\def\NN{\mathbb{N}}
\def\Z{\mathbb{Z}}
\def\ZZ{\mathbb{Z}}

\def\R{\mathbb{R}}
\def\RR{\mathbb{R}}
\def\C{\mathbb{C}}

\def\o{\omega}

\def\t{\widetilde}
\def\eps{\varepsilon}
\def\epsilon{\varepsilon}

\newcommand{\ip}[2]{\left\langle#1,#2\right\rangle}

\numberwithin{equation}{section}

\begin{document}

\title[Shearlet Multiresolution Analysis]{Adaptive Directional
Subdivision Schemes and\\ Shearlet Multiresolution Analysis}

\author{Gitta Kutyniok}
\address{Program in Applied and Computational Mathematics,
Princeton University, Princeton, NJ 08544, USA}
\email{kutyniok@math.princeton.edu}
\thanks{The first author was supported by Deutsche Forschungsgemeinschaft
(DFG) Heisenberg-Fellowship KU 1446/8-1.}

\author{Tomas Sauer}
\address{Institute of Mathematics,
        Justus--Liebig--University Gie{\ss}en,
        35392 Gie{\ss}en, Germany}
\email{tomas.sauer@math.uni-giessen.de}

\subjclass[2000]{Primary 42C40; Secondary 41A05, 42C15, 47B99,
65D10, 94A08}

\date{October 13, 2007}

\keywords{Directional transforms, joint spectral radius,
multiresolution analysis, refinement equation, shearlets,
subdivision schemes}

\begin{abstract}
In this paper, we propose a solution for a fundamental problem in
computational harmonic analysis, namely, the construction of a
multire\-so\-lution analysis with directional components. We will do
so by constructing subdivision schemes which provide a means to
incorporate directionality into the data and thus the limit
function. We develop a new type of non-stationary bivariate
subdivision schemes, which allow to adapt the subdivision process
depending on directionality constraints during its performance, and
we derive a complete characterization of those masks for which these
adaptive directional subdivision schemes converge. In addition, we
present several numerical examples to illustrate how this scheme
works. Secondly, we describe a fast decomposition associated with a
sparse directional representation system for two dimensional data,
where we focus on the recently introduced sparse directional
representation system of shearlets. In fact, we show that the
introduced adaptive directional subdivision schemes can be used as a
framework for deriving a shearlet multiresolution analysis with
finitely supported filters, thereby leading to a fast shearlet
decomposition.
\end{abstract}

\maketitle

\section{Introduction}
Efficient and economical representations of anisotropic structures
are essential in various areas in applied mathematics. The
nature of the problems we face can be divided into two types, namely
when the anisotropic structure is given explicitly and when it is
given implicitly. The analysis of images and higher dimensional data
with respect to directional features shall serve as an example of an
explicitly given anisotropic structure, whereas the solution of
hyperbolic partial differential equations often exhibits the
phenomenon of shocks which can be interpreted as an implicit
anisotropic structure.

It is well known that wavelets are perfectly suited for providing
efficient representations in the sense of sparsity for problems with
a dominant isotropic regularity, at the same time being
associated with a multiresolution analysis which is the key
ingredient for a fast decomposition algorithm. However, when dealing
with anisotropic phenomena wavelets do not perform equally well. In
fact, it can be proven that wavelets do not provide optimally sparse
representations.

In contrast to earlier approaches such as directional wavelets
\cite{antoine}, complex wave\-lets \cite{kingsbury}, ridgelets
\cite{CD99}, and contourlets \cite{DV05}, the curvelets introduced
by Cand\`{e}s and Donoho precisely satisfy this need, in the sense
of resolving the wavefront set \cite{CD05b} and the curvelet
representation being optimally sparse for objects with
$C^2$-singularities \cite{CD04}. Also there already exist some first
results on applying curvelets to hyperbolic partial differential
equations by Cand\`{e}s and Demanet \cite{CD05}. However, one
drawback is the lack of a multiresolution analysis associated with
curvelets, and, in particular, a fast decomposition algorithm in the
time domain. This raises the question about the existence of a
representation system with analyzing properties as good as
curvelets, but being equipped with a more ``wavelet-like'' structure
in the sense of being associated with a multiresolution analysis. In
fact, the discrete counterpart would then lead to finitely supported
filters that allow for a mathematically justified discrete fast
decomposition of discrete data. We anticipate such a representation
to combine the favorable computational properties of wavelets
with the main additional property to provide a means to resolve
anisotropic structures efficiently.


In this paper we give a complete, positive answer to the question of
the existence of such a system by introducing subdivision schemes for the
recently introduced concept of shearlets, thus
constructing an associated multiresolution analysis which indeed
leads to a fast discrete decomposition algorithm. The directional
representation system of {\em shearlets}
\cite{GKL06} stands out for the following reason. They do not only
precisely resolve the wavefront set \cite{KL06} and provide
optimally sparse representations \cite{GL06}, but shearlet systems
are generated by one single function which is dilated by a parabolic
scaling and a shear matrix and translated in the time domain, hence
form an affine system. We might even interpret the system of
shearlets as being generated by a strongly continuous, irreducible,
square-integrable representation of a certain group, the shearlet
group \cite{DKMSST06}. This rich mathematical structure enables, for
instance, the application of coorbit theory to study smoothness
spaces -- so-called shearlet coorbit spaces -- associated with the
decay of the shearlet coefficients \cite{DKST07}. We would further
like to mention that one attempt to associate shearlets with a so-called
generalized multiresolution analysis can be found in \cite{LLKW05}.
However, this structure did not yield a
fast decomposition due to the fact that the filters are not
compactly supported and even infinitely many filters have to be
employed.

Our approach to derive a multiresolution analysis associated with
shearlets and to provide a feasible fast shearlet decomposition
comprises the introduction of a new class of non-stationary
bivariate subdivision schemes which incorporate directionality in a
particular way.
Subdivision schemes provide a mathematical method to refine given
coarse data while
providing characterization results to ensure convergence to a
continuous function, say. Moreover, such schemes automatically
provide \emph{refinable functions} which are the basis for any
multiresolution analysis as nestedness of the different levels of
resolution is equivalent to the refinability of the underlying
``basis'' function. Homogeneous stationary subdivision schemes have been
studied extensively over the last 20 years; for an elaborate survey we
refer the reader to \cite{CavarettaDahmenMicchelli91}. Recently,
algebraic methods have been introduced as a means to derive
characterizations of convergence and approximation order
in a very natural way for multivariate subdivision (cf. \cite{Sauer02b}).
On the other hand, also the conditions of homogeneity and stationarity have
been released by various authors, leading to subdivision schemes where the
refinement rule varies with the level of iteration or the location of
refinement. However, the gain in generality always comes with the prize of a
loss of structure so that there is comparatively little known about these
generalizations (see, e.g., \cite{CD96,CC04}). In particular,
no subdivision schemes were known so far which provide a means
to adapt the subdivision process  depending on directionality
constraints during its performance while still ensuring convergence.
The development of such subdivision schemes will be important both
for construction of a shearlet multiresolution analysis as well as
for opening the research area of methods for data refinement to incorporate
anisotropic structures.


We will show in this paper that such an adaptive directional
subdivision scheme can be constructed and it will indeed lead to a
shearlet multiresolution analysis and a fast shearlet decomposition.
Our approach to derive a non-stationary bivariate adaptive
directional subdivision scheme is based on the idea to iteratively
apply two subdivision schemes each of which is associated with a
different direction. The two individual subdivision schemes can employ two
different finitely supported filters while their respective
dilation matrices are
taken from the theory of shearlet systems. We would also like to
mention at this point that the most natural ``directional''
operation, the rotation, can not be employed, since its action does
not provide a refinement of a lattice. In contract to this
observation, products of parabolic scaling and shear matrices do
indeed satisfy this desirable property. The constructed subdivision
scheme provides the opportunity to adaptively change the orientation
of the data during the subdivision process, since in each iteration
one of both single subdivision schemes can be applied. In this sense,
we can visualize the subdivision process as a binary tree, in which the
direction of the finer data is dependent on the branch we choose.
However, for convergence we certainly need to study each branch of
the tree, which requires an appropriate definition of convergence. Our first
key result shows that, provided the adaptive directional subdivision
scheme converges, we obtain associated generalized refinement
equations (Theorem \ref{T:RefEq}). These will become essential for
deriving a shearlet multiresolution analysis. As a main result we
then provide a complete characterization of those masks which lead
to convergent adaptive directional subdivision schemes (Theorem
\ref{T:Convergence}) in terms of algebraic and spectral properties of the
associated filters. In the proof we will make use of ideal theoretic methods
which come in handy to extract ``the zero at $-1$'' of the two masks.

For the construction of a shearlet multiresolution analysis we
employ the fact that each wavelet multiresolution
analysis is associated with a convergent subdivision scheme \cite{Dau92}. We
introduce scaling spaces based on the previously constructed
directional subdivision schemes, and then prove that these indeed
provide a multiresolution analysis structure (Theorem
\ref{T:ShearletMRA}) due to the refinement equations mentioned above.
This multiresolution analysis will then provide us in a very
natural way with a mathematically justified discrete fast shearlet
decomposition of discrete data which is stated as Algorithm \ref{algo:FSD}.
Also here we encounter a binary tree
structure, since the decomposition will be dependent on the
different directions which were encoded in a binary tree structure
of the subdivision process. For the construction of a shearlet
multiresolution analysis and a fast shearlet decomposition, we focus
on the situation of interpolatory masks. The non-interpolatory case
is beyond the scope of this paper and will be studied in a
forthcoming paper.


\medskip

The outline of the paper is the following. In Section
\ref{sec:refinement} we briefly introduce discrete shearlet systems.
We further study which directions can be attained by the action of
the associated dilation matrices on $\ZZ^2$. The new type of
subdivision schemes, which we baptize {\em adaptive directional
subdivision schemes}, are introduced in Section \ref{sec:adaptive}.
In Section \ref{sec:convergence} we provide a complete
characterization of convergence for those schemes along with the
necessary ideal theoretic background. Some numerical experiments on
the refinement of data employing this new type of subdivision
schemes are provided in Section \ref{sec:numerics}. We then show how
the previously derived adaptive directional subdivision schemes can
be used as a framework for deriving a {\em shearlet multiresolution
analysis} with finitely supported filters (Section
\ref{sec:shearletMRA}). In Section \ref{sec:FSD} we employ these
results to provide a {\em fast shearlet decomposition}.


\section{Refinement of $\Z^2$ by Anisotropic Scaling and Shearing}
\label{sec:refinement}


\subsection{Shearlet Dilation Matrices}
\label{sec:shearlet}

Our approach towards directional refinement of the lattice $\Z^2$
and, later on, adaptive directional subdivision schemes is inspired
by the recently introduced discrete shearlet transform \cite{GKL06},
since this transform is able to precisely detect directions of
singularities (cf. \cite{KL06}) which we will take advantage of. In
order to provide a thorough motivation for our construction, allow
us to first briefly review the idea of shearlets.

Each shearlet system forms an affine system, i.e., consists of
dilations and translations of one single generating function $\psi
\in L^2(\R^2)$, a so-called {\em shearlet}. As dilation matrices,
products of anisotropic parabolic scaling matrices and shear
matrices -- which coined the name ``shearlets'' -- are employed. In
order to define a shearlet system, let $A_a$, $a > 0$, and $S_s$, $s
\in \R$, which are defined by
\[ A_a = \begin{pmatrix}
     a & 0\\ 0 & \sqrt{a}
 \end{pmatrix}
 \quad \mbox{and} \quad
S_s = \begin{pmatrix}
     1 & -s\\ 0 & 1
 \end{pmatrix},\]
denote a {\em parabolic scaling matrix} and a {\em shear matrix},
respectively. Then the {\em shearlet system} associated with a
shearlet $\psi \in L^2(\R^2)$ is given by
\begin{equation}\label{eq:shearlet}
\{ \psi_{jkm}(x) :=  2^{-\frac32j} \psi(S_{-k} A_{4^{-j}} x - m) :
j,k \in \Z,\,m \in \Z^2\}.
\end{equation}
The three parameters $j,k,m$ are interpreted in the following way:
$j$ provides  the scale, and $k$ and $m$ detect the direction and
position of singularities, respectively. It is easy to construct
shearlets such that \eqref{eq:shearlet} forms a Parseval frame for
$L^2(\R^2)$, for instance, by choosing
$\hat{\psi}(\xi_1,\xi_2)=\hat{\psi_1}(\xi_1)
\hat{\psi_2}(\xi_2/\xi_1)$, where $\psi_1 \in L^2(\RR)$ is a
discrete wavelet, i.e., $\sum_{j \in \ZZ} |\hat{\psi}_1(4^j \o)|^2 =
 1$ for  $\o \in \RR$, satisfying $\hat{\psi}_1 \in C^\infty(\RR)$ and
$\supp \, \hat{\psi}_1 \subset [-1,-\frac14] \cup [\frac14,1]$, and
$\psi_2 \in L^2(\RR)$ is a bump function satisfying $\hat{\psi}_2
\in C^\infty(\RR)$, $\supp \,\hat{\psi}_2 \subset [-1,1]$, and
$\sum_{k \in \ZZ} |\hat{\psi}_2(k+\o)|^2 = 1$ for $\o \in \RR$ (cf.
\cite{GKL06}). The associated {\em Shearlet Transform}
$\mathcal{S}\mathcal{H}_\psi$ is then defined on $L^2(\RR^2)$ by
\[ \mathcal{S}\mathcal{H}_\psi f (j,k,m) = \ip{f}{\psi_{jkm}}.\]

In order to provide an equal treatment of the direction of the $x$-
and $y$-axis, the frequency plane is split into the cone
\[C = \{(\xi_1, \xi_2) \in \R^2: \, |\xi_1| \ge \tfrac 14, \, |\tfrac{\xi_2}{\xi_1}| \le 1\},\]
its by $90^0$ rotated copy, and the square centered at the origin of
side length $\frac12$. The Shearlet Transform acts on $C$ and its
copy as described above, while the choice of $\psi$ has to be
adapted appropriately. The center square can be filled in such a way
that this system also forms a Parseval frame. The shearlet system in
$C$ and its copy is usually referred to as {\em shearlets on the
cone}, see \cite{GKL06}. The associated tiling of the frequency
plane is illustrated in Figure \ref{fig:ShearletsOnCone}.

\begin{figure}[h]
\begin{center}
  \ifpdf
  \includegraphics[width=5cm]{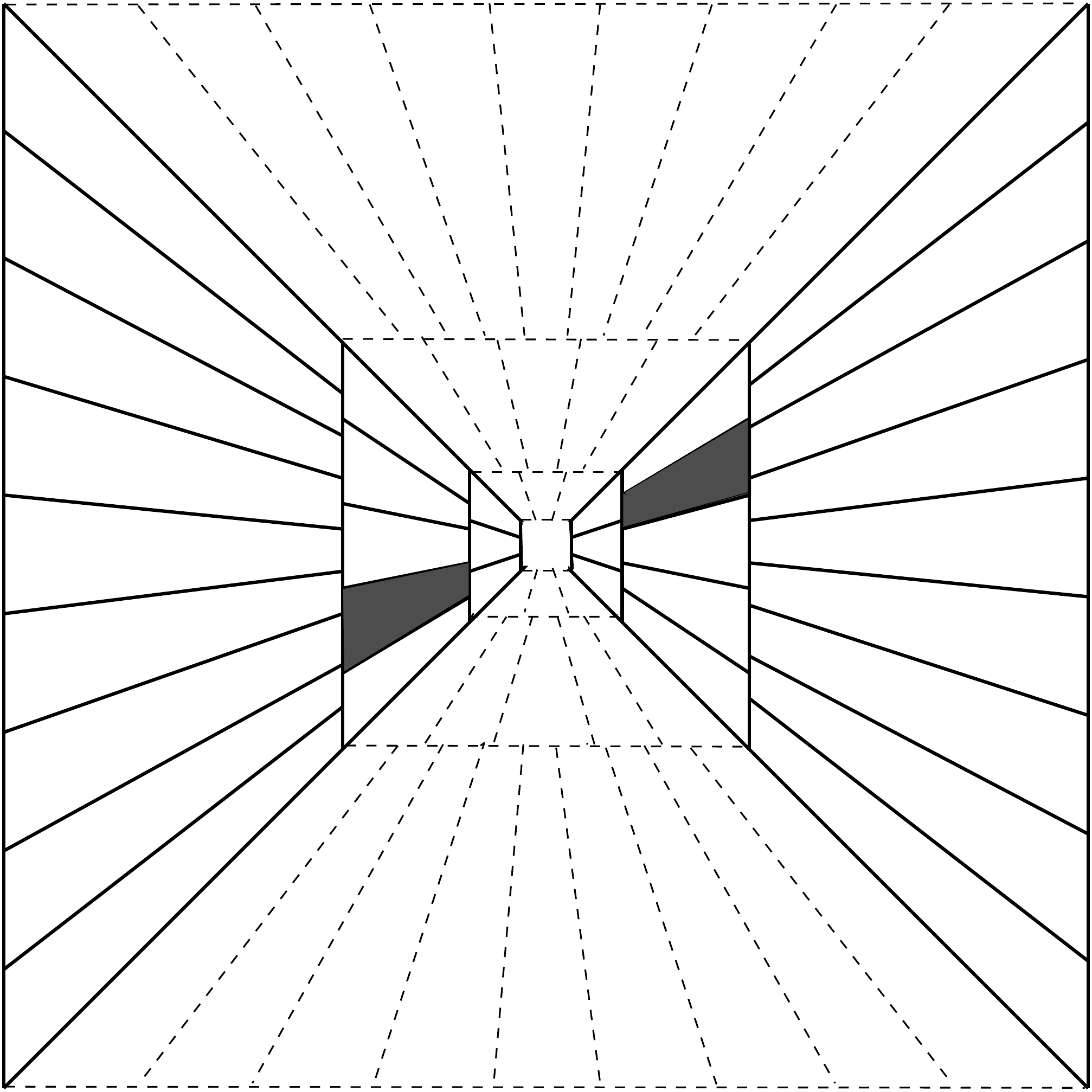}
  \else
  \includegraphics[width=5cm]{ShearletsOnCone.eps}
  \fi
\end{center}
\caption{The tiling of the frequency domain induced by the shearlets
on the cone.} \label{fig:ShearletsOnCone}
\end{figure}

The refinement matrices interesting to us for deriving a directional
refinement of the lattice $\ZZ^2$ are the dilation matrices used in
\eqref{eq:shearlet} for $j=1$, i.e., the matrices
\[M_{k}:=S_{-k} A_{\frac14}
=
 \begin{pmatrix}
     1 & k\\ 0 & 1
 \end{pmatrix}
 \begin{pmatrix}
     \frac14 & 0\\ 0 & \frac12
 \end{pmatrix}
 =
 \begin{pmatrix}
     \frac14 & \frac12k\\ 0 & \frac12
 \end{pmatrix},
 \quad k \in \Z.\]
Following the philosophy of the shearlets on the cone, also the
matrices
\[ \widetilde{M}_{k}:=
\begin{pmatrix}  \frac12 & 0  \\
\\ \frac12 k  & \frac14  \end{pmatrix},\]
which serve as dilation matrices for the rotated copy of $C$, will
be employed as refinement matrices. The matrices $M_{k}$ and
$\widetilde{M}_{k}$ not only provide the possibility to map a line
to various directions, but moreover possess the property of refining
the lattice $\Z^2$ equally at each level as it is shown in the
following result.

\begin{proposition}
\label{prop:matrixrefinement} The following conditions hold.
\begin{enumerate}
\item For all $j,k \in \ZZ$, we have
\[M_{k} \in \mbox{GL}_2(\R) \quad \mbox{and} \quad
M_{k}(4^{-j} \Z \times 2^{-j} \Z) = 4^{-(j+1)} \Z \times 2^{-(j+1)}
\Z.\]
\item For all $j,k \in \ZZ$, we have
\[\widetilde{M}_{k} \in \mbox{GL}_2(\R) \quad \mbox{and} \quad
\widetilde{M}_{k}(2^{-j} \Z \times 4^{-j} \Z) = 2^{-(j+1)} \Z \times
4^{-(j+1)} \Z.\]
\end{enumerate}
\end{proposition}

\begin{proof}
(i) The first claim is obvious. To prove the second claim, let $j,k
\in \Z$ and $m=(m_1,m_2) \in \Z^2$. Then
\[M_{k}
 \begin{pmatrix}
    4^{-j} m_1\\2^{-j} m_2
 \end{pmatrix}
 =
 \begin{pmatrix}
    4^{-(j+1)} m_1 + 2^{-(j+1)}k m_2\\2^{-(j+1)} m_2
 \end{pmatrix}
 =
 \begin{pmatrix}
    4^{-(j+1)} (m_1 + 2^{j+1}k m_2)\\2^{-(j+1)} m_2
 \end{pmatrix},\]
which implies $M_{k}(4^{-j} \Z \times 2^{-j} \Z) \subseteq
4^{-(j+1)} \Z \times 2^{-(j+1)} \Z$.

Now let $n=(n_1,n_2) \in \Z^2$. Then choosing $m = (m_1,m_2) \in
\Z^2$ as $m_1=n_1-2^{j+1}kn_2$ and $m_2 = n_2$ yields
\[M_{k}
 \begin{pmatrix}
    4^{-j} m_1\\2^{-j} m_2
 \end{pmatrix}
 =
 \begin{pmatrix}
    4^{-(j+1)} (n_1-2^{j+1}kn_2 + 2^{j+1}k n_2)\\2^{-(j+1)} n_2
 \end{pmatrix}
 =
\begin{pmatrix}
    4^{-(j+1)}n_1\\2^{-(j+1)}n_2
 \end{pmatrix}.\]
 Thus $M_{k}(4^{-j} \Z \times 2^{-j} \Z) \supseteq 4^{-(j+1)} \Z \times 2^{-(j+1)} \Z$,
 which proves the claim.

(ii) This follows by using similar arguments as in part (i).
\end{proof}

Thus, when applying a sequence of matrices $M_{k_1}, \ldots,
M_{k_n}$ iteratively to the lattice $\Z^2$, at the $j$th level the
points $\{(4^{-j}(m_1+\ell \frac14),2^{-j}(m_2+\frac12)) : \ell \in
\{1,2,3\}\}$ are added to the lattice $4^{-j} \Z \times 2^{-j} \Z$.
This is true for an arbitrary choice of integers $k_j \in \ZZ$, $1
\le j \le n$. Moreover, at each level this map is bijective.

A similar result holds for the matrices $\widetilde{M}_k$, $k \in
\ZZ$.


\subsection{Feasible Directions}
\label{sec:feasible}

Let us now delve deeper into the explicit construction of the
refinement by using the splitting idea of the shearlets on the cone.
The overall aim is to provide a way of refinement such that the
points on the $y$-axis -- or any other line through the origin --
can be moved to an arbitrary line through the origin during the
refinement process. This immediately forces the refinement scheme to
provide different strategies for refinement. We will see how this is
can be achieved by using the matrices $M_{\eps}$ and
$\widetilde{M}_{\eps}$ even only for $\eps =-1,0,1$. In the sequel
we will only focus on the matrices $M_{\eps}$, $\eps=-1,0,1$, since
the others can be treated simultaneously.

In the very first step of the refinement, we apply $M_{\eps}$ to
$\Z^2$ for $\eps=-1,0,1$. Application of $\eps=0$ does not change
any directions, $\eps=1$ maps the $y$-axis to the angle bisector in
the first and third quadrant of the plane, and $\eps=-1$ has the
same effect on the second and fourth quadrant. From now on, we
consider the two cases $\eps \in \{0,-1\}$ or $\eps \in \{0,1\}$
separately. Focusing on the second case, in each step we not only
derive the refinement from a coarser scale $4^{-j} \Z \times 2^{-j}
\Z$ to a finer scale $4^{-(j+1)} \Z \times 2^{-(j+1)} \Z$, but also
have two different ways to achieve this, either by applying $M_0$
or by applying $M_1$. Hence, at the $n$th level we have applied a
product of the form $M_{\eps_n} \ldots M_{\eps_1}$ to $\ZZ^2$, where
$\eps_j \in \{0,1\}$ for each $1 \le j \le n$. For $\eps \in
\{0,-1\}$, one can proceed in exactly the same way which we will,
however, not work out in detail in this paper.

From now on, we will use the abbreviation $E_n = \{0,1\}^n$, $n \in
\NN$, for the index sets and will also denote by
\[
E = \bigcup_{n \in \NN} E_n
\]
the set of all finite $0$-$1$--sequences and by $E_\infty =
\{0,1\}^\NN$ the space of all infinite sequences. Note that $E$ is
canonically embedded in $E_\infty$ by the mapping
\[
E \ni \eps \mapsto \eps^* = \left( \eps,0,0,\dots \right) \in
E_\infty.
\]

The main question to ask at this point concerns the possible
directions this procedure allows us to map the points on the
$y$-axis to. For this analysis, we restrict our attention to the
first quadrant of the plane, since the same refinements occur in the
third quadrant only in an origin-symmetric way.

\smallskip

We first notice that the sequence of $n$ matrices $M_\eps$ we choose
is completely determined by the associated sequence $\eps \in E_n$.
Hence this refinement scheme has the structure of a binary tree as
illustrated in Figure \ref{fig:binary}.

\begin{figure}[h]
\begin{center}
  \begin{picture}(250,190)(0,0)
\put(0,90){$\ZZ^2$} \put(15,100){\vector(1,1){45}}
\put(15,90){\vector(1,-1){45}} \put(65,140){$M_0 \ZZ^2$}
\put(65,40){$M_1 \ZZ^2$} \put(95,150){\vector(2,1){45}}
\put(95,140){\vector(2,-1){45}} \put(95,50){\vector(2,1){45}}
\put(95,40){\vector(2,-1){45}} \put(145,170){$M_0 M_0 \ZZ^2, \quad
\eps=(0,0)$} \put(145,113){$M_0 M_1 \ZZ^2, \quad \eps=(1,0)$}
\put(145,70){$M_1 M_0 \ZZ^2, \quad \eps=(0,1)$} \put(145,13){$M_1
M_1 \ZZ^2, \quad \eps=(1,1)$}
  \end{picture}
\end{center}
\caption{The binary tree up to level $2$ associated with the
refinement scheme.} \label{fig:binary}
\end{figure}
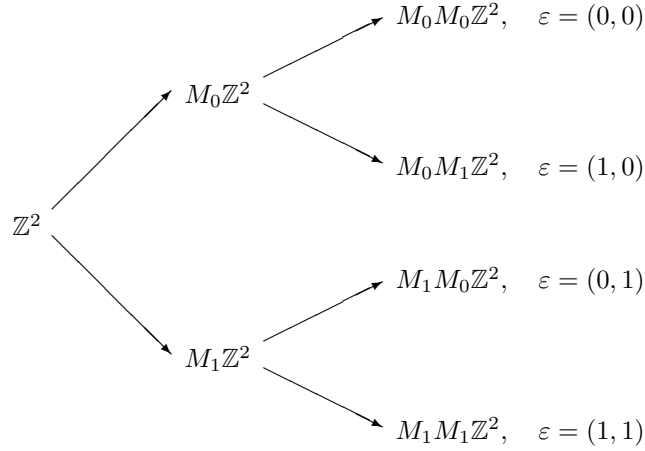

The directions which might be obtained employing this refinement
scheme are encoded in this binary tree in a special though natural
way. To explore this relation, we first compute the product of the
matrices which is applied to achieve the refinement at level $n$.
Interestingly, the following binary number appears therein.

\begin{notation}
For $\eps \in E_n$, $n \in \NN$, we define
\[ \left( \eps \right)_2 = \sum_{j=0}^{n-1} \eps_{j+1} \, 2^j
\qquad\mbox{and}\qquad M_\eps = M_{\eps_n} \cdot \ldots  \cdot
M_{\eps_1}.
\]
\end{notation}

\noindent Using this notion we obtain the following form for a
refinement matrix $M_\eps$.

\begin{lemma}
\label{lemma:productsofM} Let $n \in \N$ and $\eps \in E_n$. Then we
have
\[
M_\eps =
\begin{pmatrix}
  4^{-n}&  4^{-n} \, 2 \, (\eps)_2  \\
  0 & 2^{-n}
 \end{pmatrix}.\]
\end{lemma}

\begin{proof}
We will prove this lemma by induction. For $n=1$, the claim
obviously holds. Now suppose that the claim is true for some $n \in
\N$. Let $\eps = \left( \eps',\eps_{n+1} \right) \in E_{n+1}$,
$\eps' \in E_n$. We have to distinguish between $\eps_{n+1} = 0$,
hence $\left( \eps \right)_2 = \left( \eps' \right)_2$, with
\[ M_\eps = M_{\left( \eps',\eps_{n+1} \right)} =
\begin{pmatrix}
  4^{-1}&  0\\
  0 & 2^{-1}
\end{pmatrix}
\begin{pmatrix}
  4^{-n}&  4^{-n} \,  2  \, \left( \eps' \right)_2  \\
  0 & 2^{-n}
\end{pmatrix}
= \begin{pmatrix}
  4^{-(n+1)}&  \frac12 \,  4^{-n} \, (\eps')_2  \\
  0 & 2^{-(n+1)}
\end{pmatrix},
\]
and $\eps_{n+1} = 1$, i.e., $(\eps)_2 = \left( \eps' \right)_2 +
2^{n+1}$, where
\begin{eqnarray*}
M_\eps & = &
\begin{pmatrix}
    4^{-1}&  2^{-1}\\
    0 & 2^{-1}
 \end{pmatrix}
\begin{pmatrix}
  4^{-n}&  4^{-n} \,  2 \,  (\eps')_2  \\
  0 & 2^{-n}
\end{pmatrix}
= \begin{pmatrix}
  4^{-(n+1)}&  \frac12 \left( 4^{-n} (\eps')_2 + 2^{-n}
  \right)  \\
  0 & 2^{-(n+1)}
\end{pmatrix}\\
& = & \begin{pmatrix}
  4^{-(n+1)}&  \frac12  \, 4^{-n} \,  \left( (\eps')_2+2^{n} \right)  \\
  0 & 2^{-(n+1)}
\end{pmatrix}
= \begin{pmatrix}
  4^{-(n+1)}&  \frac12  \, 4^{-n}  \, (\eps)_2  \\
    0 & 2^{-(n+1)}
 \end{pmatrix}.
\end{eqnarray*}
which advances the induction hypothesis.
\end{proof}

\begin{notation}
Let $L$ be a line through the origin and $\eps \in E_n$, $n \in \N$.
Then $s(L,\eps)$ denotes the slope of $M_{\eps}L$, which is again a
line through the origin. We further write $s(L)$ for the slope of $L$.
\end{notation}

The next result computes the values of the slopes $s(L,\eps)$.

\begin{lemma}
Let $L$ be a line through the origin and $\eps \in E_n$, $n \in \N$.
  Then the following relations between $s(L,\eps)$, $\eps$
  and the original $L$ hold.
  \begin{enumerate}
  \item If $L$ is a line through the origin with $s(L) \in (0,\infty)$, then
    \[
    s(L,\eps) = \frac{2^{n}}{\frac{1}{s(L)}+2(\eps)_2}.
    \]
  \item If $L = \{0\} \times \R$, i.e., $s(L) = \infty$, then
    \[ s(L,\eps)    = \frac{2^{n-1}}{(\eps)_2},
    \]
    where we set $2^{n-1}/0 := \infty$.
  \item If $L = \R \times \{0\}$, i.e., $s(L) = 0$, then
    \[
    s(L,\eps)  = 0.
    \]
\end{enumerate}
\end{lemma}

\begin{proof}
(i) We consider the point $(1,s(L)) \in L$. Using Lemma
\ref{lemma:productsofM}, we compute
\[  M_{\eps}
\begin{pmatrix}
    1\\
    s(L)
 \end{pmatrix}
= \begin{pmatrix}
    4^{-n}&  4^{-n}  \, 2 \, (\eps)_2 \\
    0 & 2^{-n}
 \end{pmatrix}
\begin{pmatrix}
    1\\
    s(L)
 \end{pmatrix}
=\begin{pmatrix}
    4^{-n} (1+2\,s(L)\,(\eps)_2)  \\
    2^{-n}s(L)
 \end{pmatrix}
\]
Hence, the slope of the line $M_{\eps}L$ equals
\[\frac{4^{n}\,s(L)}{2^{n}(1+2\,s(L)\,(\eps)_2)}
= \frac{2^{n}\,s(L)}{1+2\,s(L)\,(\eps)_2} =
\frac{2^{n}}{\frac{1}{s(L)}+2\,(\eps)_2}.\]

(ii) Here we consider the point $(0,1) \in L=\{0\} \times \R$. Again
employing Lemma \ref{lemma:productsofM}, we obtain
\[ M_{\eps}
\begin{pmatrix}
    0\\
    1
 \end{pmatrix}
= \begin{pmatrix}
    4^{-n}&  4^{-n} \, 2  \, (\eps)_2  \\
    0 & 2^{-n}
 \end{pmatrix}
\begin{pmatrix}
    0\\
    1
 \end{pmatrix}
=\begin{pmatrix}
    4^{-n} \,  2  \, (\eps)_2  \\
    2^{-n}
 \end{pmatrix}.
\]
Thus
\[s(L,\eps) = \frac{4^{n}}{2^{n+1}(\eps)_2}
= \frac{2^n}{2 \,(\eps)_2}.\]

(iii) is easily verified by noting that the point $(1,0)$ is mapped
to
\[
M_{\eps}
\begin{pmatrix}
    1\\
    0
  \end{pmatrix}
  = \begin{pmatrix}
    4^{-n}&  4^{-n}  \, 2 \,  (\eps)_2  \\
    0 & 2^{-n}
  \end{pmatrix}
  \begin{pmatrix}
    1\\
    0
 \end{pmatrix}
 =
 \begin{pmatrix}
   4^{-n} \\ 0
 \end{pmatrix}
\]
so that the slope remains zero.
\end{proof}

Our main result in this section will show that indeed the points on
an arbitrary line through the origin of slope $\neq 0$ can be moved
arbitrarily close to prescribed lines through the origin during the
refinement process.

\begin{theorem}
\label{theo:possibledirections} Let $L$ be a line through the origin
with $s(L) \in (0,\infty]$. Then, for each $t \in
[\frac12,\infty]$ and $\delta > 0$, there exists some $n \in \NN$
and $\eps \in E_n$ such that
\[ |s(L,\eps)-t| < \delta.\]
\end{theorem}

\begin{proof}
Suppose $L$ is a line through the origin with $s(L) \in (0,\infty)$. The
case $s(L) = \infty$ can be dealt with in a similar way.

For given $t \in (\frac12,\infty)$ and $\delta > 0$, due to the
denseness of rational numbers there exists some $n \in \NN$ and
$\eps \in E_n$ such that
\[ \left| \sum_{j=0}^{n-1} \eps_{j+1} \, 2^{j-n+1} - \frac{1}{t}\right| <
\frac{\delta}{t(t+\delta)} =: \t{\delta}.
\]
Indeed, $\eps$ can be chosen as a truncation of the binary expansion
of $1/t$. Note that without loss of generality we can assume that
\[ \frac{1}{2^n s(L)} < \t{\delta},\]
since we can always enlarge $n$. Using these relations, we obtain
\begin{eqnarray*}
\lefteqn{
  |s(L,\eps)-t|
  = \left|\frac{2^{n}}{\frac{1}{s(L)}+2(\eps)_2} - t\right| } \\[1ex]
& = & \left|\frac{1}{\frac{1}{2^n s(L)}+\sum_{j=0}^{n-1} \eps_{j+1} \,
2^{j-n+1}} - t\right| = \left|\frac{1-t(\frac{1}{2^n
s(L)}+\sum_{j=0}^{n-1} \eps_{j+1} \, 2^{j-n+1})}{\frac{1}{2^n s(L)}
+\sum_{j=0}^{n-1} \eps_{j+1} \, 2^{j-n+1}}\right| \\[1ex]
& \le & t \left|\frac{\frac{1}{t}-\sum_{j=0}^{n-1} \eps_{j+1} \,
2^{j-n+1} -\frac{1}{2^n s(L)}}{\frac{1}{t}-\t{\delta}}\right| \le
\left|\frac{t^2\t{\delta}}{1-t\t{\delta}}\right| = \delta.
\end{eqnarray*}
Note that for the last equality we used $\t{\delta} < \frac{1}{t}$.

Now let $\eps \in E_\infty$ be defined by $\eps_j = 0$ for all $j
\ge j_0$ for some $j_0 \in \NN$, and let $M > 0$. Then there exists
some $n \in \NN$ such that
\[\frac{1}{2^n s(L)}+\sum_{j=0}^{n-1} \eps_{j+1} \, 2^{j-n+1} < \frac{1}{M}
\quad \mbox{for all }n \ge n_0,\] which implies
\[s(L,(\eps_1 \ldots \eps_n)) = \frac{1}{\frac{1}{2^n s(L)}+\sum_{j=0}^{n-1} \eps_{j+1} \, 2^{j-n+1}} > M,\]
hence $\lim_{n \to \infty} s(L,(\eps_1 \ldots \eps_n)) = \infty$.

Finally, let $\eps \in E_\infty$ be defined by $\eps_j = 1$ for all
$j \ge j_0$ for some $j_0 \in \NN$. Then, for all $n \in \NN$,
\[s(L,(\eps_1 \ldots \eps_n))
= \frac{2^{n}}{\frac{1}{s(L)}+\sum_{j=1}^{n} \eps_{j} \, 2^j} =
\frac{2^{n}}{\frac{1}{s(L)}+2^{n+1}-2-\sum_{j=1}^{j_0-1} (1-\eps_{j})
\, 2^j}\] and hence,
\[
  \lim_{n \to \infty} s(L,(\eps_1 \ldots \eps_n))
  = \lim_{n \to \infty} \frac{1}{\frac{1}{2^{n} s(L)}+2-\frac{1}{2^{n-1}}
    -\frac{1}{2^{n}}\sum_{j=1}^{j_0-1} (1-\eps_{j}) \, 2^{j}} =
\frac12. \qedhere
\]


\end{proof}

Thus only employing $M_0$ and $M_1$ we can move any line arbitrarily
close to any line of slope $\in [\frac12,\infty]$. This shows the
range of directions we might attain (compare Figure
\ref{fig:SplittingOfPlane}). However, we would like to mention that
the change of orientation of the data induced by the subdivision
scheme (see Definition \ref{def:adaptivesubdivision}) is also
affected by directionality of the masks.


\smallskip


\begin{theorem}
  \label{theo:possibledirections2}
  Let $L$ be a line through the origin with $s \in (0,\infty]$. Then, for each
  $t \in [-\frac12,-\infty]$ and $\delta > 0$, there exists some
  $n \in \NN$ and $\eps \in E_n$ such that
  \[ |s(L,\eps)-t| < \delta.\]
\end{theorem}

Similar results as Theorems \ref{theo:possibledirections} and
\ref{theo:possibledirections2} also hold for the matrices
$\t{M}_\eps$, $\eps \in \{-1,0,1\}$. We omit to also state these
results for the sake of brevity, since they are similar to the
previous theorems.


\subsection{A Directional Refinement of the Lattice $\Z^2$}
\label{sec:refinementoflattice} The results in the preceding section
point out how to refine $\Z^2$ in a directional way such that all
possible directions can be attained. Dependent on whether we intend
to map say the $y$-axis to a line with a slope contained in $[\frac12,\infty]$,
$[-\frac12,-\infty]$, or $[-\frac12,\frac12]$, we choose to refine
by using the matrices $M_0,M_1$, $M_{-1},M_0$, or
$\t{M}_{-1},\t{M}_{0},\t{M}_{1}$, respectively. Once the type of
matrices is chosen, we iterate depending on the angle we would like
to attain by using Theorem \ref{theo:possibledirections}, Theorem
\ref{theo:possibledirections2}, or the corresponding result for the
matrices $\t{M}_\eps$, $\eps \in \{-1,0,1\}$.
For an illustration of the different areas of lines through the
origin which can be attained during the refinement process dependent
on the chosen matrices we refer to Figure
\ref{fig:SplittingOfPlane}.

\begin{figure}[h]
\begin{center}
\begin{picture}(300,180)(0,0)
\put(70,0)%
{\ifpdf
\includegraphics[width=6cm]{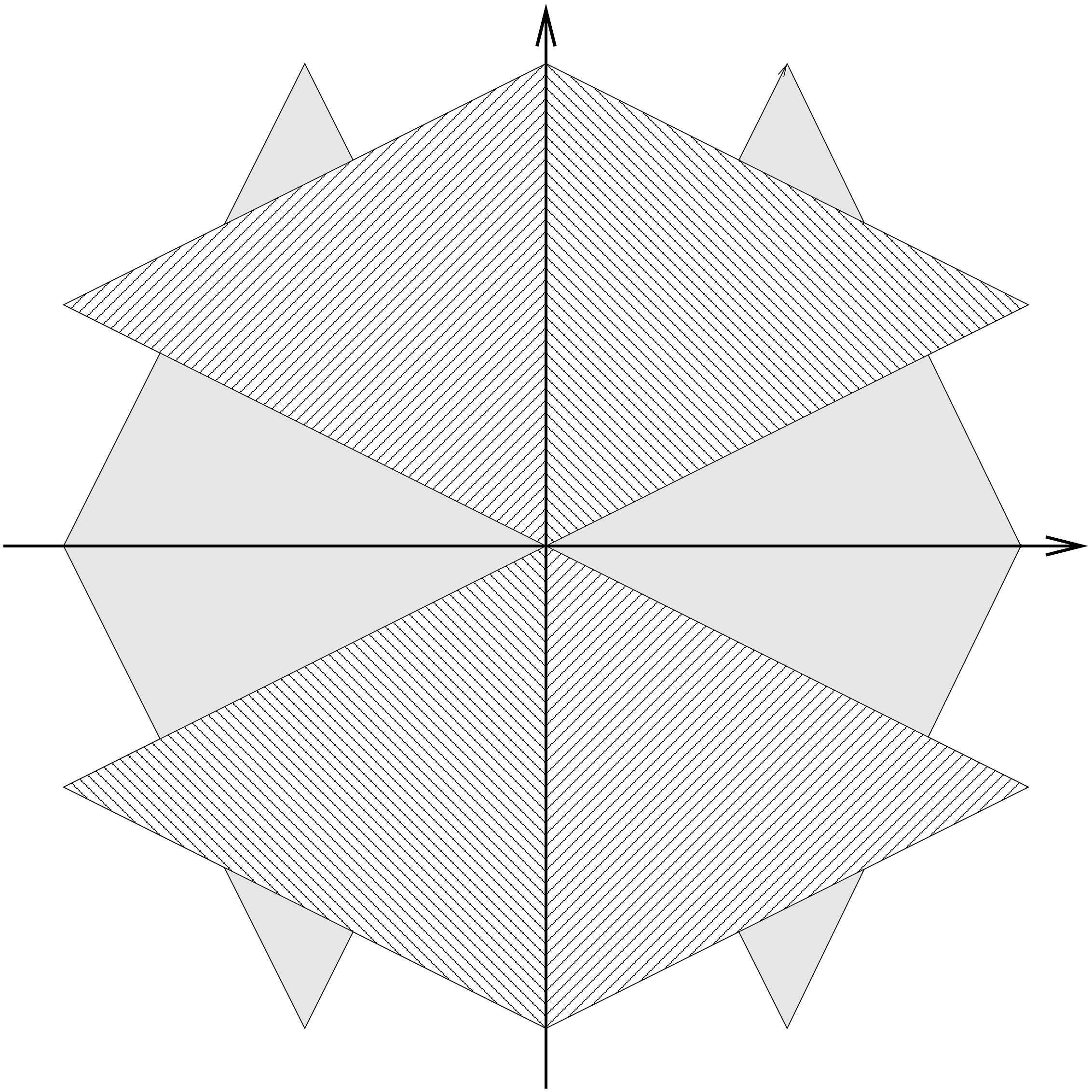}
\else
\includegraphics[width=6cm]{Areas1.eps}
\fi}
\put(230,150){\vector(-1,-1){40}} \put(235,155){$M_\eps \mbox{ with
} \eps=0,1$} \put(235,140){(cf. Theorem
\ref{theo:possibledirections})} \put(250,108){\vector(-1,0){50}}
\put(255,104){$\mbox{Slope } \frac12$}
\put(80,150){\vector(1,-1){40}} \put(-20,155){$M_\eps \mbox{ with }
\eps=-1,0$} \put(-20,140){(cf. Theorem
\ref{theo:possibledirections2})} \put(230,65){\vector(-2,1){40}}
\put(235,55){$\t{M}_\eps \mbox{ with } \eps=-1,0,1$}
  \end{picture}
\end{center}
\caption{This figure shows the different areas of lines through the
origin which can be attained during the refinement process depending
on the choice of $M_\eps$ and $\widetilde{M}_\eps$ and $\eps \in
\{-1,0,1\}$.} \label{fig:SplittingOfPlane}
\end{figure}

From now on we will focus entirely on the matrices $M_0$ and $M_1$.
All following results can be derived in a similar way for
$M_{-1},M_0$ and for $\t{M}_{-1},\t{M}_{0},\t{M}_{1}$.


\section{Adaptive Directional Subdivision}
\label{sec:adaptive} In this section, we finally arrive at the
announced definition of a new type of subdivision schemes, based on
the interaction of \emph{two} ``normal'' stationary subdivision
schemes, which we will study in the sequel. To that end, we choose
two \emph{masks} $a_\eps$, $\eps \in \{0,1\}$, i.e., \emph{finitely
supported} sequences $a_\eps \in \ell_{00} \left( \Z^2 \right)$ as
well as the expanding scaling matrices $W_\eps = M_\eps^{-1}$, $\eps
\in \{0,1\}$. These matrices can be given explicitly as
\begin{equation}\label{eq:W0andW1}
W_0 =
\begin{pmatrix}
  4 & 0 \\ 0 & 2
\end{pmatrix}
\qquad \mbox{and} \qquad W_1 =
\begin{pmatrix}
  4 & -4 \\ 0 & 2
\end{pmatrix},
\end{equation}
and again we set $W_\eps = W_{\eps_n} \cdots W_{\eps_1}$, $\eps \in
E_n$. Also note that
\[
W_1 = W_0 \,
\begin{pmatrix}
  1 & -1 \\ 0 & 1
\end{pmatrix}
=
\begin{pmatrix}
  1 & -2 \\ 0 & 1
\end{pmatrix}
\, W_0.
\]
Such a decomposition also exists for the iterated matrices $W_\eps$,
$\eps \in \{0,1\}^n$, $n \in \N$.

To formulate the next auxiliary result, we also define for $\eps \in
E_n$ the dyadic number
\[
\left[ \eps \right]_2 = .\eps_1 \ldots \eps_n := \sum_{j=1}^n \eps_j
\, 2^{-j} \in [0,1].
\]
With this notation at hand, we obtain the following counterpiece of
Lemma~\ref{lemma:productsofM}.

\begin{lemma}\label{lemma:productsofW}
  For $n \in \N_0$ and $\eps \in E_n$, we have
  \[
  W_\eps = W_{\eps_n} \cdots W_{\eps_1} =
  \begin{pmatrix}
    4^{n} & -4^{n} \, 2 \, [\eps]_2 \\
    0 & 2^{n}
  \end{pmatrix}
  = U_\eps \, W_0^n  = W_0^n V_\eps,
  \]
  where
  \[
  U_\eps =
  \begin{pmatrix}
    1 & -2^{n+1} \left[ \eps \right]_2 \\
    0 & 1
  \end{pmatrix}
  \quad \mbox{and} \quad
  V_\eps =
  \begin{pmatrix}
    1 & -2 \left[ \eps \right]_2 \\ 0 & 1
  \end{pmatrix},
  \]
  hence $U_\eps = V_\eps^{2^n}$.
\end{lemma}

\begin{proof}
  The proof is again of inductive nature and relies on noting that
  \[
  W_0 W_\eps =
  \begin{pmatrix}
    4 & 0 \\ 0 & 2
  \end{pmatrix}
  \begin{pmatrix}
    4^{n} & -4^{n} \, 2 \, [\eps]_2 \\
    0 & 2^{n+1}
  \end{pmatrix} =
  \begin{pmatrix}
    4^{n+1} & -4^{n+1} \, 2 \, \left[ \left( \eps,0 \right) \right]_2 \\
    0 & 2^{n+2}
  \end{pmatrix}
  \]
  as well as
  \begin{eqnarray*}
  W_1 W_\eps & = &
  \begin{pmatrix}
    4 & -4 \\ 0 & 2
  \end{pmatrix}
  \begin{pmatrix}
    4^{n} & -4^{n} \, 2 \, [\eps]_2 \\
    0 & 2^{n}
  \end{pmatrix}
  =
  \begin{pmatrix}
    4^{n+1} & -4^{n+1} \left( 2 \, [\eps]_2 + 2^{-n} \right) \\
    0 & 2^{n+1}
  \end{pmatrix} \\
  & = &
  \begin{pmatrix}
    4^{n+1} & -4^{n+1} \, 2 \, \left[ \left( \eps,1 \right) \right]_2 \\
    0 & 2^{n+1}
  \end{pmatrix}.
  \end{eqnarray*}
  Hence,
  \[
  W_\eps =
  \begin{pmatrix}
    2^{2n} & -2^{2n+1} [\eps]_2 \\ 0 & 2^n
  \end{pmatrix}
  = W_0^n \,
  \begin{pmatrix}
    1 & -2 [\eps]_2 \\ 0 & 1
  \end{pmatrix}
  =
  \begin{pmatrix}
    1 & -2^{n+1} [\eps]_2 \\ 0 & 1
  \end{pmatrix} \, W_0.
  \]
  Since for $x \in \R$
  \[
  \begin{pmatrix}
    1 & -x \\ 0 & 1
  \end{pmatrix}^k
  =
  \begin{pmatrix}
    1 & -k x \\ 0 & 1
  \end{pmatrix},
  \]
  also the final claim follows.
\end{proof}

Note that $V_{(0,\dots,0)}$, $V_{(1,0,\dots,0)}$, and all $U_\eps$
are \emph{unimodular} matrices, i.e., they have an inverse in $\Z^{2
\times 2}$. A particular role will be played by the two matrices
\[
V =
\begin{pmatrix}
  1 & -1 \\ 0 & 1
\end{pmatrix},
\qquad U = V^2 =
\begin{pmatrix}
  1 & -2 \\ 0 & 1
\end{pmatrix}
\]
which satisfy
\begin{equation}
  \label{eq:UVProps}
  W_1 = U W_0 = W_0 V, \qquad \mbox{i.e.} \qquad
  W_1 = U^{-1} W_1 V \qquad \mbox{and} \qquad
  W_0 = U W_0 V^{-1}.
\end{equation}

The associated subdivision schemes are now defined as follows. The
term {\em adaptive} refers to the tree-like structure, which
provides various branches for subdivision, whereas the term {\em
directional} refers to the directional structure which comes from
the shearing process contained in the dilation matrices $W_\eps$,
$\eps \in E$.

\begin{definition} \label{def:adaptivesubdivision}
Let $a_\eps \in \ell_{00}(\ZZ^2)$, $\eps \in \{0,1\}$ be two masks,
that is, two \emph{finitely supported sequences}, and let $W_\eps$,
$\eps \in \{0,1\}$ be defined as in \eqref{eq:W0andW1}. Then the
associated {\em adaptive directional subdivision scheme of order
$n$} is defined by
\[
S_\eps = S_{\eps_n} \cdots S_{\eps_1}, \qquad \eps \in E_n, \qquad n
\in \NN,
\]
where, for $\eta \in \{0,1\}$,
\[S_\eta c := S_{a_\eta,W_\eta} c
:= \sum_{\alpha \in \Z^2} a_\eta \left( \cdot - W_\eta \, \alpha
\right) \, c \left( \alpha \right), \quad c \in \ell_\infty \left( \ZZ^2 \right),
\]
\end{definition}
Note that both the mask as well as the scaling matrix of these
subdivision schemes depend on the index $\eps$. Moreover, we wish to
remark that these schemes can clearly be computed in a tree--like
fashion by setting
\[
S_\eps c = S_{(\eps',\eps_n)} c = S_{\eps_n} S_{\eps'} = \sum_{\beta \in
\Z^2} a_{\eps_n} \left( \cdot - W_{\eps_n}
    \beta \right) \, S_{\eps'} c (\beta),
\qquad \eps' \in E_{n-1}.
\]

Adaptive directional subdivision schemes can be considered
subdivision schemes of their own, however, with a different scaling
matrix. This is easily seen by means of the following example: for
$\alpha \in \Z^2$ we have {\allowdisplaybreaks
\begin{eqnarray*}
  S_{\left( \eps_1,\eps_2 \right)} c & = &
  S_{\eps_2} S_{\eps_1} c
  = \sum_{\beta \in \Z^2} a_{\eps_2} \left( \cdot - W_{\eps_2} \beta
  \right) \left( S_{\eps_1} c \right) (\beta) \\
  & = & \sum_{\beta \in \Z^2} a_{\eps_2} \left( \cdot - W_{\eps_2}
    \beta \right)
  \sum_{\gamma \in \Z^2} a_{\eps_1} \left( \beta - W_{\eps_1} \gamma
  \right) \, c (\gamma) \\
  & = & \sum_{\gamma \in \Z^2} \left[ \sum_{\beta \in \Z^2}
    a_{\eps_2} \left( \cdot - W_{\eps_2} \beta - W_{\eps_2}
      W_{\eps_1} \gamma \right) \, a_{\eps_1} (\beta) \right]
  \, c (\gamma) \\
  & =: & \sum_{\gamma \in \Z^2} a_{\left( \eps_1,\eps_2 \right)}
  \left( \cdot - W_{\left( \eps_1,\eps_2 \right)} \gamma \right)
  \, c (\gamma).
\end{eqnarray*}
} An inductive application of this argument immediately gives the
next result.

\begin{lemma}\label{lemma:IterMasks}
  For $\eps \in E_n$, the subdivision scheme $S_\eps$ acts as
  \begin{equation*}
    S_\eps c \, (\alpha)  = \sum_{\beta \in \Z^2}
    a_\eps \left( \alpha - W_\eps \beta \right) \, c (\beta), \qquad
    \alpha \in \Z^2,
  \end{equation*}
  where the coefficient sequences $a_{\eps}$ are recursively defined as
  $a_{\eps} =  a_{\left( \eps',\eps_n \right)}
  = S_{\eps_n} a_{\eps'}$.
\end{lemma}

To get a better understanding of the geometry of adaptive
directional subdivision, we write $a_1$ as $a_1 = \widetilde a_0
\left( U \cdot \right)$ which is always possible since $U$ is
unimodular. It then follows from repeated applications of
(\ref{eq:UVProps}) that
{\allowdisplaybreaks
\begin{eqnarray*}
  S_{a_1,W_1} c & = &
  \sum_{\alpha \in \Z^2} a_1 \left( \cdot - W_1 \alpha \right)
  \, c (\alpha)\\
  & = & \sum_{\alpha \in \Z^2} \widetilde a_0
  \left( U \cdot - U W_1 U^{-1} U \alpha \right) \, c (\alpha) \\
  & = & \sum_{\alpha \in \Z^2} \widetilde a_0 \left( U \cdot -
    U W_1 V^{-2} \alpha \right) \, c \left( U^{-1} \alpha \right)\\
  & = & \sum_{\alpha \in \Z^2} \widetilde a_0 \left( U \cdot - U W_0 V^{-1} \alpha
  \right) \, c \left( U^{-1} \alpha \right) \\
  & = & \sum_{\alpha \in \Z^2} \widetilde a_0 \left( U \cdot - W_0 \alpha
  \right) \,
  c \left( U^{-1} \alpha \right)\\
  & = & \left( S_{\widetilde a_0,W_0} c \left( U^{-1} \cdot \right) \right)
  \left( U \cdot \right).
\end{eqnarray*}
}
This identity can be rewritten in terms of dilation operators as
\begin{equation*}
  S_1 = D_U \, \widetilde S_0 \, D_{U^{-1}} = D_U \, \widetilde S_0 \, D_U^{-1},
  \mbox{ hence} \quad
  S_{(1,\dots,1)} = D_U \, \widetilde S_{(0,\dots,0)} \, D_U^{-1},
\end{equation*}
and enables us to implement the subdivision scheme $S_1$ in terms of
$\widetilde S_0$ and the shear operator $D_U$. Moreover, it explains
the geometry of the scheme $S_1$: first, a shearing by $U^{-1}$ is
applied to the data sequence, then the subdivision operator refines
the data in the sheared direction with a higher resolution than the
data in the non--sheared direction, so that the additional
application of the shearing by $U$ does not fully compensate the
initial one. In summary, this process leads to limit functions which
are sheared versions of the limit function of $S_0$ and the amount
of shearing is determined by when and how often $S_1$ is applied in
the process. We remark that this geometry is very much in the spirit
of the Continuous Shearlet Transform, which can be regarded as
applying a shearing operator, an anisotropic 2-D Wavelet Transform,
and again a shearing operator \cite{KS07}.

\section{Convergence}
\label{sec:convergence}

In this section, we shall study convergence of the previously
introduced adaptive directional subdivision schemes. To that end, we
introduce the \emph{projection operators} $P_n \;:\; E_\infty \to
E_n$, $n \in \NN$, which extract the initial segment of order $n$
from a sequence: $P_n \eps = \left( \eps_1,\dots,\eps_n \right)$.

\begin{definition}
\label{def:convergence}
  The adaptive directional subdivision scheme is said to be
  \emph{convergent in $C \left( \R^2 \right)$}, if, for any
  $\eps \in E_\infty$, there exists a nonzero uniformly continuous
  function $f_\eps \in C \left( \R^2 \right)$
  such that
  \begin{equation*}
    \lim_{n \to \infty} \sup_{\alpha \in \Z^2}
    \left| f_\eps \left( W_{P_n \eps}^{-1} \alpha \right) -
        S_{P_n \eps} \delta (\alpha) \right|
    = 0.
  \end{equation*}
  Note that this is equivalent to
  \begin{equation*}
    \lim_{n \to \infty} \sup_{\alpha \in \Z^2}
    \left| f_\eps \left( W_{P_n \eps}^{-1} \alpha \right) -
      a_{P_n \eps} (\alpha) \right| = 0.
  \end{equation*}
\end{definition}

\noindent Since any sequence $c \in \ell \left( \Z^2 \right)$ can be
trivially written as
$$
c = \sum_{\alpha \in \Z^2} c(\alpha) \, \delta \left( \cdot - \alpha
\right), \qquad \delta (\alpha) := \delta_{\alpha,0},
$$
and since the subdivision operator is linear, we immediately obtain
the following convolution style representation of the limit
function.

\begin{proposition}
  If the adaptive directional subdivision scheme converges for some
  $\eps \in E_\infty$ then
  the limit function takes the form
  $$
  f_\eps * c =
  \sum_{\alpha \in \Z^2}  c(\alpha) \, f_\eps \left( \cdot - \alpha \right).
  $$
\end{proposition}

\subsection{Basic Properties}
This definition of convergence has an immediate consequence: If the
adaptive directional subdivision scheme is a convergent one, then,
in particular, $a_0$ and $a_1$ must define convergent adaptive
directional subdivision schemes, which follows by simply choosing
$\eps = (0,0,\dots)$ and $\eps = (1,1,\dots)$, respectively.
Consequently, they must both preserve constants.

\begin{lemma}\label{lemma:subdivisionconvergent}
  If the adaptive directional subdivision scheme is convergent, then
  \begin{equation}
    \label{eq:ConSumRule0}
    \sum_{\beta \in \Z^2} a_\eps \left( \alpha + W_\eps \beta \right)
    = 1, \qquad \alpha \in \Z^2, \qquad \eps \in \{0,1\}.
  \end{equation}
\end{lemma}

An alternative but equivalent definition of convergence of a
adaptive directional subdivision scheme can be given in terms of
function spaces instead of sequence spaces by means of test
functions.

\begin{definition}
  A function $g \in C \left( \R^2 \right)$ is called a \emph{test function},
  if it is compactly supported and its integer translates form a stable
  partition of unity, that is,
  \begin{enumerate}
  \item $\sum_\alpha g \left( \cdot - \alpha \right) = 1$,
  \item there exist constants $0 < A < B < \infty$ such that for any
    $c \in \ell_\infty$
    \[
    A \, \left\| c \right\|_\infty \le \left\| g * c \right\|_\infty \le
    B \| c \|_\infty, \qquad g*c := \sum_{\alpha \in \Z^2} c(\alpha)
    \, g \left( \cdot - \alpha \right).
    \]
  \end{enumerate}
\end{definition}

The most prominent examples for test functions are the tensor
product B--Splines so that there even exist \emph{refinable} test
functions of arbitrary regularity. With the help of test functions,
convergence can be described as follows.

\begin{theorem}\label{T:AlterConvDesc}
  The adaptive directional subdivision scheme converges if and only if for any
  $\eps \in E_\infty$ there exists a nonzero uniformly continuous function
  $f_\eps$ such that
  \begin{equation}
    \label{eq:AlterConvDesc}
    \lim_{n \to \infty} \left\| f_\eps - \left( g * S_{P_n \eps}
        \delta \right) \left( W_{P_n \eps} \right) \right\|_\infty = 0
  \end{equation}
  \begin{enumerate}
  \item for some test function $g$.
  \item for any test function $g$.
  \end{enumerate}
\end{theorem}

\begin{proof}
  For classical subdivision, this result is due to Dahmen and Micchelli
  \cite{DahmenMicchelli97} and we just show how it can be extended in
  a straightforward way to adaptive directional subdivision. To that end, let
  $g$ be any test function and recall that for any \emph{uniformly
    continuous} function $f$ and any expanding matrix $M$ the
  ``quasi-interpolant''
  \[
  g * \sigma_{M} f = \sum_{\alpha \in \Z^2} f \left( M \, \alpha
  \right) \, g (\cdot - \alpha),
  \]
  with the \emph{sampling operator}
  $\sigma_{M} := \left( f { \left( M \, \alpha \right) } \;:\;
    \alpha \in \Z^2 \right)$, satisfies
  \begin{equation*}
    \left\| f - g * \sigma_{M^{-1}} f \left( M \cdot \right)
    \right\|_\infty \le C_g
    \, \omega \left( f, \left\| M^{-1} \right\| \right),
  \end{equation*}
  where
  \[
  \omega \left( f,\delta \right) := \sup_{x \in \R^2} \sup_{\| x-y \|_\infty
    \le \delta} \left| f(x) - f(y) \right|,
  \]
  denotes the modulus of continuity of $f$. Recall that
  $\omega \left( f,\delta \right) \to 0$ for $\delta \to 0$ as long as $f$
  is uniformly continuous. Now, we have that
  \begin{eqnarray*}
    \lefteqn{
      \left\| f_\eps - \left( g * S_{P_n \eps}
          \delta \right) \left( W_{P_n \eps} \cdot \right) \right\|_\infty
    } \\
    & \le & \left\| f_\eps - \left( g * \sigma_{W_{P_n \eps}^{-1}}
        f_\eps \right)
      \left( W_{P_n \eps} \cdot \right) \right\|_\infty
    + \left\| g * \left( \sigma_{W_{P_n \eps}^{-1}} f_\eps - S_{P_n
          \eps} \delta \right) \left( W_{P_n \eps} \cdot
        \right) \right\|_\infty \\
    & = & \left\| f_\eps - \left( g * \sigma_{W_{P_n \eps}^{-1}}
        f_\eps \right)
      \left( W_{P_n \eps} \cdot \right) \right\|_\infty
    + \left\| g * \left( \sigma_{W_{P_n \eps}^{-1}} f_\eps - S_{P_n
          \eps} \delta \right) \right\|_\infty \\
    & \le & C_g \,
    \omega \left( f, \left\| W_{P_n \eps}^{-1} \right\| \right) +
      B \, \left\| \sigma_{W_{P_n \eps}^{-1}} f_\eps - S_{P_n
          \eps} \delta \right\|_\infty.
  \end{eqnarray*}
  On the other hand,
  \begin{eqnarray*}
    \lefteqn{
      \left\| \sigma_{W_{P_n \eps}^{-1}} f_\eps - S_{P_n
          \eps} \delta \right\|_\infty
      \le A^{-1} \left\| g * \left( \sigma_{W_{P_n \eps}^{-1}}
          f_\eps - S_{P_n \eps} \delta \right)
        \left( W_{P_n \eps} \cdot
        \right) \right\|_\infty
    } \\
    & \le & A^{-1} \left( \left\| f_\eps - \left( g *
        \sigma_{W_{P_n \eps}^{-1}} f_\eps \right)
      \left( W_{P_n \eps} \cdot \right) \right\|_\infty
    + \left\| f_\eps - \left( g * S_{P_n \eps} \delta \right)
        \left( W_{P_n \eps} \cdot
        \right) \right\|_\infty \right) \\
    & \le & A^{-1} \left( C_g \,
    \omega \left( f, \left\| W_{P_n \eps}^{-1} \right\| \right) +
    \left\| f_\eps - \left( g * S_{P_n \eps} \delta \right)
        \left( W_{P_n \eps} \cdot
        \right) \right\|_\infty \right)
  \end{eqnarray*}
  which verifies the equivalence. Since therefore convergence of the
  adaptive directional subdivision scheme is equivalent to \eqref{eq:AlterConvDesc} holding for
  an arbitrary test function, this property holds for one particular
  test function if and only if it holds for any test function.
\end{proof}

\begin{theorem}\label{T:RefEq}
  If the adaptive directional subdivision scheme converges, then the limit functions
  $f_\eps$, $\eps \in E_\infty$, satisfy the
  \emph{refinement equation}
  \begin{equation}
    \label{eq:RefEq}
    f_\eps = \sum_{\alpha \in \Z^2} a_{\eps_1} (\alpha) \,
    f_{\widehat \eps} \left( W_{\eps_1} \cdot - \alpha \right),
    \qquad \widehat\eps := \left( \eps_2,\eps_3,\dots \right).
  \end{equation}
\end{theorem}

\begin{proof}
  We define the \emph{transition operator}
  \[
  T_\eps f = \sum_{\alpha \in \Z^2} a_\eps (\alpha) \, f \left(
    W_\eps \cdot - \alpha \right), \qquad f \in C \left( \R^2 \right),
    \; \eps \in \{0,1\}
  \]
  and note that,  for $c \in \ell_\infty$,
  \begin{eqnarray*}
    \left( T_\eps f \right) * c
    & = & \sum_{\alpha \in \Z^2} T_\eps f \left( \cdot - \alpha
    \right) \, c (\alpha)
    = \sum_{\alpha,\beta \in \Z^2} a_\eps (\beta) \, c (\alpha)
    f \left( W_\eps \cdot - W_\eps \alpha - \beta \right) \\
    & = & \sum_{\beta \in \Z^2} \left( \sum_{\alpha \in \Z^2} a_\eps
      \left( \beta - W_\eps \alpha \right) \, c(\alpha) \right)
    \, f \left( W_\eps \cdot - \beta \right)
    = \left( f * S_\eps c \right) \left( W_\eps \cdot \right).
  \end{eqnarray*}
  By iteration, we then find for $\eps \in \{0,1\}^n$ that
  \begin{eqnarray*}
    \left( f * S_\eps c \right) \left( W_\eps \cdot \right)
    & = & \left( f * S_{\eps_n} \cdots S_{\eps_1} c \right)
    \left( W_{\eps_n} \cdot \ldots \cdot W_{\eps_1} \cdot \right) \\
    & = & \left( T_{\eps_n} f * S_{\eps_{n-1}} \cdots
      S_{\eps_1} c \right)
    \left( W_{\eps_{n-1}} \cdot \ldots \cdot W_{\eps_1} \cdot \right)
    = \ldots = \left( T_\eps f * c \right)
  \end{eqnarray*}
  where
  \[
  T_\eps f = T_{\eps_1} \cdots T_{\eps_n} f, \qquad \eps
  \in \{0,1\}^n.
  \]
  Since, for $n \in \N$,
  \begin{eqnarray*}
    \lefteqn{
      T_{\eps_1} f_{\widehat \eps}
      = T_{\eps_1} f_{\widehat \eps} * \delta
      = \left( f_{\widehat \eps} * S_{\eps_1} \delta \right)
      \left( W_{\eps_1} \cdot \right) } \\
    & = & \left[ \left( f_{\widehat \eps} - \left( g * S_{P_{n-1}
            \widehat \eps}  \delta \right) \left( W_{P_{n-1} \widehat
            \eps}
        \right) \right) * S_{\eps_1} \delta \right]
    \left( W_{\eps_1} \cdot \right)\\
&&    + \left[ \left( g * S_{P_{n-1} \widehat \eps} \delta \right)
\left(
        W_{P_{n-1} \widehat \eps} \right) * S_{\eps_1} \delta \right]
    \left( W_{\eps_1} \cdot \right)
    \\
    & = & \left[ \left( f_{\widehat \eps} - \left( g * S_{P_{n-1}
            \widehat \eps} \delta  \right)
        \left( W_{P_{n-1} \widehat \eps}
        \right) \right) * S_{\eps_1} \delta \right]
    \left( W_{\eps_1} \cdot \right) + \left( g * S_{P_n \eps} \delta
    \right) \left( W_{P_n \eps} \cdot \right),
  \end{eqnarray*}
  it follows that
  \begin{eqnarray*}
    \lefteqn{\left\| T_{\eps_1} f_{\widehat \eps} - f_\eps
    \right\|_\infty}\\
    & \le &
    \left\| \left( f_{\widehat \eps} - \left( g * S_{P_{n-1} \widehat
            \eps} \right) \left( W_{P_{n-1} \widehat \eps} \right)
      \right) * S_{\eps_1} \delta \right\|_\infty
    + \left\| f_{\eps} - \left( g * S_{P_n \eps} \delta
      \right) \left( W_{P_n \eps} \cdot \right) \right\|_\infty \\
    & \le & \left\| S_{\eps_1} \right\| \,
    \left\| f_{\widehat \eps} - \left( g * S_{P_{n-1} \widehat
          \eps} \right) \left( W_{P_{n-1} \widehat \eps} \right)
    \right\|_\infty
    + \left\| f_{\eps} - \left( g * S_{P_n \eps} \delta
      \right) \left( W_{P_n \eps} \cdot \right) \right\|_\infty
  \end{eqnarray*}
  and the right hand side of this inequality converges to zero for
  $n \to \infty$ while the left hand side is independent of $n$. Thus
  $T_{\eps_1} f_{\widehat \eps} = f_\eps$ which is
  (\ref{eq:RefEq}).
\end{proof}

\subsection{An Algebraic Description, Sum Rules and Polynomial Reproduction}
Next, we give a more detailed description of the necessary condition
(\ref{eq:ConSumRule0}) from Lemma~\ref{lemma:subdivisionconvergent}
in algebraic terms. To that end, we recall the definition of the
\emph{symbol} of a mask $a$, defined as
\[
a^* (z) = \sum_{\alpha \in \Z^2} a(\alpha) \, z^\alpha, \qquad z \in
\C_*^2 = \left( \C \setminus \{0\} \right)^2,
\]
as well as the \emph{subsymbols}
\[
a_{\eps,\eta}^* (z) = \sum_{\alpha \in \Z^2} a \left( \eta + W_\eps
  \alpha \right) \, z^\alpha, \qquad \eta \in H_\eps := W_\eps^T \,
\left[ 0,1 \right)^2 \cap \Z^2, \qquad \eps \in \{0,1\}.
\]
The symbol can be ``reconstructed'' from the subsymbols by the
well--known formula
\begin{equation*}
  a^* (z) = \sum_{\eta \in H_\eps} z^\eta \, a_{\eps,\eta}^* \left(
    z^{W_\eps} \right), \qquad \eps \in \{0,1\},
\end{equation*}
from which the following result follows immediately, cf.
\cite{Sauer02b}.

\begin{proposition}
  The mask $a_\eps$ satisfies (\ref{eq:ConSumRule0}), the
  \emph{sum rule of order $0$}, if and only if
  \[
  a^* (z) = 0, \qquad z \in \left\{ e^{-2\pi i W_\eps^{-T} \eta} \;:\;
    \eta \in H_\eps \setminus \{0\} \right\}.
  \]
\end{proposition}

For a more algebraic description, we need the notion of a
\emph{quotient ideal}. Recall that an ideal in $\Lambda$, the ring
of \emph{Laurent polynomials} in two variables, is a subset of
$\Lambda$ that is closed under addition and multiplication by
arbitrary Laurent polynomials. The \emph{quotient ideal} of two
Laurent ideals $I,J$, is defined as
\[
I : J := \left\{ f \in \Lambda \;:\; f \cdot J \subseteq I \right\}
\]
and has the almost obvious property that $I \subseteq I : J$. For
any matrix $X \in \Z^{2 \times 2}$, with column vectors $x_1,x_2$ we
finally define the ideal
\[
\left\langle z^X - 1 \right\rangle := \left\langle z^{x_1} -
1,z^{x_2} - 1 \right\rangle := \left\{ f_1 (z) \left( z^{x_1} - 1
\right) + f_2 (z) \left( z^{x_2} - 1
  \right) \;:\; f_1, f_2 \in \Lambda \right\}
\]
and its special case $\left\langle z - 1 \right\rangle :=
\left\langle z^I - 1 \right\rangle$. Then we have the following
result from \cite{MoellerSauer04}.

\begin{theorem}\label{T:QuotId}
  The mask $a_\eps$ satisfies (\ref{eq:ConSumRule0}), the
  \emph{sum rule of order $0$}, if and only if
  \[
  a^* \in \left\langle
    z^{W_\eps} - 1 \right\rangle : \left\langle z - 1 \right\rangle.
  \]
\end{theorem}

\noindent To conveniently formulate an important consequence of this
theorem, we introduce the vectors
\[
\left[ z^X - 1 \right] = \left[
  \begin{array}{c}
    z^{x_1} - 1 \\ z^{x_2} - 1
  \end{array}
\right], \qquad X = \left[ x_1,x_2 \right] \in \ZZ^{2 \times 2}.
\]
With this notation we have the following result.

\begin{corollary}\label{C:QIdRep}
  If the adaptive directional subdivision scheme converges, then there exist
  \emph{matrix valued masks} $B_\eps$, $\eps \in \{0,1\}$
  such that
  \begin{equation}
    \label{eq:QIdRep}
    \left[ z-1 \right] \, a_{\eps}^* (z) = B_\eps^* (z) \, \left[
      z^{W_\eps} - 1 \right], \qquad \eps \in \{0,1\}.
  \end{equation}
\end{corollary}

\begin{proof}
  Any convergent subdivision must satisfy the sum rule of order $0$ for
  $a_\eps$, $\eps \in \{0,1\}$, and so, by Theorem~\ref{T:QuotId},
  it follows for $\eps \in \{0,1\}$ and $j=1,2$ that
  \[
  \left( z_j - 1 \right) \, a_\eps^* (z)
  = b_{j1}^* (z) \, \left( z^{\left( W_\eps \right)_1} - 1 \right)
  + b_{j2}^* (z) \, \left( z^{\left( W_\eps \right)_2} - 1 \right).
  \]
  Written in matrix form, this is what has been claimed.
\end{proof}

\begin{definition}
  The matrix masks $B_\eps$, $\eps \in \{0,1\}$, from
  (\ref{eq:QIdRep}) are called
  \emph{representation masks} of $a_\eps$, $\eps \in \{0,1\}$,
  respectively.
\end{definition}

\begin{remark}
  Recall that the computation of the representation masks $B_\eps$
  can be performed by \emph{reduction}, a multivariate generalization of
  division with remainder, see \cite{CoxLittleOShea92,Sauer01}
  for the term order and homogeneous versions of this process, respectively.
  Therefore, the symbolic determination of $B_\eps$ can easily be done
  with the help of practically any Computer Algebra system that supports
  constructive polynomial ideal theory.

  Note however, that the representation masks are \emph{not} unique
  to the appearance of \emph{syzygies} of $\left[ z^{W_\eps} - 1 \right]$,
  not even if an H--representation, cf. \cite{MoellerSauer00}, is
  chosen where -- in the case of $W_0$ -- we have the ``minimal degree''
  requirements that
  \[
  \deg b_{11} = \deg b_{21} = \deg a_0 - 3, \quad
  \deg b_{12} = \deg b_{22} = \deg a_0 - 1,
  \]
  see also \cite{Sauer02}.
\end{remark}

\noindent We continue by giving explicit bases of the quotient
ideals for our specific choice of $W_\eps$. This is easy for $W_0$
as all entries in this matrix are nonnegative, and indeed it is not
difficult to see that
\begin{eqnarray*}
  I_0 & := & \left\langle z^{W_0} - 1
    \right\rangle : \left\langle z-1 \right\rangle
    = \left\langle z_1^4 - 1, z_2^2 - 1 \right\rangle :
    \left\langle z-1 \right\rangle \\
    & = & \left\langle \left( z_1^3 + z_1^2 + z_1 + 1 \right) \left( z_2 + 1
      \right)\right\rangle + \left\langle z_1^4 - 1 \right\rangle
    + \left\langle z_2^2 - 1 \right\rangle.
\end{eqnarray*}
In fact, the graded homogeneous leading terms of the above ideal
basis are $z_1^3 z_2$, $z_1^4$ and $z_2^2$ so that the quotient
space is spanned exactly by the seven monomials
$$
1, z_1, z_2, z_1^2, z_1 z_2, z_1^3, z_1^2 z_2,
$$
and their number coincides with the number of joint zeros of $I_0$.
Hence, by the same reasoning as in \cite{MoellerSauer04,Sauer02}
they even form a graded Gr\"obner basis, hence an H--basis of the
ideal $I_0$. Recall that a subset $H$ of an ideal $I$ is called an
\emph{H--basis}, if any polynomial $f \in I$ can be written in the
form
\[
f = \sum_{h \in H} f_h \, h, \qquad \deg f \ge \deg f_h + \deg h,
\]
where $\deg$ denotes, as usual, the \emph{total degree} of a
polynomial. We will also use $\Pi_n$ for the vector space of all
polynomials of total degree at most $n$.

The situation for $I_1 = \left\langle z^{W_1} - 1 \right\rangle$
appears to be a little bit more intricate due to the appearance of a
negative entry in $W_1$. Here it is helpful to recall that $W_1 = U
W_0$, $U =
\begin{pmatrix}
  1 & -2 \\ 0  & 1
\end{pmatrix}$, to define $y = z^U = \left( z_1, z_1^{-2} z_2 \right)$,
hence also $z = y^{U^{-1}} = \left( y_1, y_1^2 y_2 \right)$ and to
realize that
\begin{eqnarray*}
  I_1 & = & \left\langle z^{W_1} - 1 \right\rangle : \left\langle z - 1
  \right\rangle
  = \left\langle z^{U W_0} - 1 \right\rangle : \left\langle z - 1
  \right\rangle
  = \left\langle y^{W_0} - 1 \right\rangle : \left\langle y^{U^{-1}} - 1
  \right\rangle
\end{eqnarray*}
Since
\begin{eqnarray*}
  \left\langle y^{U^{-1}} - 1 \right\rangle
  & = & \left\langle y_1 - 1, y_1^2 y_2 - 1 \right\rangle
  = \left\langle y_1 - 1, y_1^2 y_2 - \left( y_1 y_2 + y_2\right)
    \left( y_1 - 1 \right) - 1 \right\rangle \\
  & = & \left\langle y_1 - 1, y_2 - 1 \right\rangle
  = \left\langle y - 1 \right\rangle,
\end{eqnarray*}
we thus obtain that
\begin{eqnarray*}
  I_1 & = & \left\langle y^{W_0} - 1 \right\rangle : \left\langle y - 1
  \right\rangle\\
  & = & \left\langle \left( y_1^3 + y_1^2 + y_1 + 1 \right) \left( y_2 + 1
    \right)\right\rangle + \left\langle y_1^4 - 1 \right\rangle
  + \left\langle y_2^2 - 1 \right\rangle \\
  & = & \left\langle \left( z_1^3 + z_1^2 + z_1 + 1 \right) \left( z_2
      + z_1^2 \right)\right\rangle + \left\langle z_1^4 - 1 \right\rangle
  + \left\langle z_2^2 - z_1^4 \right\rangle
\end{eqnarray*}
To arrive at the somewhat surprising observation that in fact $I_1 =
I_0$, we add $z_1^4 - 1$ to the third basis element, $z_2^2 -
z_1^4$, yielding $z_2^2 - 1$ again, and subtract $\left( z_1 + 1
\right) \left( z_1^4 - 1 \right)$ from the first basis element which
leads to
\begin{eqnarray*}
  \lefteqn{
    \left( z_1^3 + z_1^2 + z_1 + 1 \right) \left( z_2 + z_1^2 \right)
    - \left( z_1 + 1 \right) \left( z_1^4 - 1 \right) } \\
  & = & \left( z_1^3 + z_1^2 + z_1 + 1 \right) z_2 +
  z_1^5 + z_1^4 + z_1^3 + z_1^2 - z_1^5 - z_1^4 + z_1 + 1 \\
  & = & \left( z_1^3 + z_1^2 + z_1 + 1 \right) \left( z_2 + 1
  \right)
\end{eqnarray*}
and therefore to the following result.

\begin{theorem}\label{T:QuotIdBasis}
  The two quotient ideals $I_\eps = \left\langle z^{W_\eps} - 1
  \right\rangle : \langle z-1 \rangle$, $\eps \in \{0,1\}$, coincide and have the H--basis
  representation
  \begin{equation}
  \label{eq:QuotIdBasis}
    I := I_0 = I_1
    = \left\langle \left( z_1^3 + z_1^2 + z_1 + 1 \right) \left( z_2 + 1
      \right)\right\rangle + \left\langle z_1^4 - 1 \right\rangle
    + \left\langle z_2^2 - 1 \right\rangle.
  \end{equation}
\end{theorem}

\noindent The fact that $I_0 = I_1$ may appear a little bit
surprising at first view, since it implies that, for any finitely
supported mask $a$, we have
\[
\sum_{\beta \in \ZZ^2} a \left( \alpha + W_0 \beta \right) = 1,
\quad \alpha \in \ZZ^2 \qquad \Leftrightarrow \qquad \sum_{\beta \in
\ZZ^2} a \left( \alpha + W_1 \beta \right) = 1, \quad \alpha \in
\ZZ^2.
\]
Hence the necessary ``sum rule'' condition with respect to $W_0$ is
equivalent to the one with respect to $W_1$. However, if we write
$W_1 = W_0 V$ with the unimodular matrix $V =
\begin{pmatrix}
  1 & -1 \\ 0 & 1
\end{pmatrix}$, then a simple change of the summation variable indeed
gives for any $\alpha \in \ZZ^2$
\[
\sum_{\beta \in \ZZ^2} a \left( \alpha + W_1 \beta \right) =
\sum_{\beta \in \ZZ^2} a \left( \alpha + W_0 V \beta \right)
= \sum_{\beta \in \ZZ^2} a \left( \alpha + W_0 \beta \right),
\]
and confirms (\ref{eq:QuotIdBasis}).

Moreover, note that Theorem~\ref{T:QuotIdBasis} gives a way to
\emph{parameterize} the ideal of all admissible polynomial masks.
Indeed, for any $n \in \NN$ we have that
\begin{eqnarray*}
  I \cap \Pi_d & = & p (z) \, \left( z_1^4-1 \right) + q (z) \,
  \left( z_1^3 + z_1^2 + z + 1 \right) \left( z_2 + 1 \right)
  + r(z) \, \left( z_2^2 - 1 \right), \\
  & & \qquad \mbox{with }\deg p \le n-4, \, \deg q \le n-4, \mbox{ and }\deg r \le n-2.
\end{eqnarray*}
For a polynomial of this form, the decomposition with respect to
$W_0$, i.e., the matrix polynomial $B_0$, becomes
\begin{equation}
  \label{eq:B0Compute}
  B_0^* (z) =
  \begin{pmatrix}
    \left( z_1 - 1 \right) \, p(z) + \left( z_2 + 1 \right) \, q(z)
    & \left( z_1 - 1 \right) \, r(z) \\
    \left( z_2 - 1 \right) \, p(z) &  \left( z_1^3 + z_1^2 + z_1 + 1 \right)
    \, q(z) + \left( z_2 - 1 \right) \, r(z).
  \end{pmatrix}
\end{equation}
Since the two Laurent ideals $I_0$ and $I_1$ coincide, the
decomposition of $a_1^*$ into $B_1^*$ takes exactly the same form as
$B_0^*$ in (\ref{eq:B0Compute}).

Next, we rephrase the identity (\ref{eq:QIdRep}) by means of the
\emph{backwards difference operator} $\nabla$, defined for a
sequence $a$ as
\[
\nabla a :=
\begin{pmatrix}
  a \left( \cdot - \eta_1 \right) - a (\cdot) \\ a \left( \cdot -
    \eta_2 \right) - a (\cdot)
\end{pmatrix}, \qquad
\left( \nabla a \right)^* (z) = \left[ z - 1 \right] \, a^* (z),
\]
where $\eta_1 =
\begin{pmatrix}
  1 \\ 0
\end{pmatrix}
$ and $\eta_2 =
\begin{pmatrix}
  0 \\ 1
\end{pmatrix}
$ denote the unit multiindices in $\Z^2$. Since, in addition, any
finitely supported matrix sequence $B$ satisfies
\[
\left( S_{B,W_\eps} c \right)^* (z) = B^* (z) \, c \left( z^{W_\eps}
\right), \quad c \in \nabla \ell_\infty(\ZZ^2), \; \eps \in \{0,1\},
\]
where
\[ S_{B,W_\eps} c := \sum_{\alpha \in \Z^2} B \left( \cdot - W_\eps \, \alpha
\right) \, c \left( \alpha \right),
\]
our quotient ideal representation (\ref{eq:QIdRep}) can equivalently
be written in terms of the difference operator as
\begin{equation*}
  \nabla S_{a_\eps,W_\eps} = S_{B_\eps,W_\eps} \nabla, \qquad \eps \in \{0,1\}.
\end{equation*}

We end this section by recalling that quotient ideal containment
also characterizes the order of \emph{polynomial reproduction}
provided by the two masks and thus the subdivision scheme. Recall
that a mask $a$ provides polynomial reproduction of order $n$, if
the leading forms of all polynomial sequences are reproduced by the
scheme:
\[
S_a \Pi_k = \Pi_k, \quad k = 0,\dots,n, \qquad \Pi_k := \left\{
\sum_{|\gamma| \le k}
  a_\gamma \, \alpha^\gamma \;:\; \alpha \in \ZZ^2 \right\}.
\]
Polynomial reproduction is essential for the smoothness of the
refinable limit function \cite{CavarettaDahmenMicchelli91} as well
as for the approximation order of the associated wavelet
construction. With the methods from \cite{MoellerSauer04,Sauer02} we
can now easily describe polynomial reproduction.

\begin{theorem}
  The directional subdivision scheme preserves polynomials of degree $n$,
  i.e., $S_\eps \Pi_k = \Pi_k$, $\eps \in E$, $k=0,\dots,n$, if and only if
  \[
  a_\eps \in I^{n+1} = \left( \left\langle z^{W_0} - 1 \right\rangle :
    \left\langle z - 1 \right\rangle \right)^{n+1}
  = \left\langle z^{W_0} - 1 \right\rangle^{n+1} : \left\langle z - 1
  \right\rangle^{n+1}.
  \]
\end{theorem}

\subsection{A Characterization of Convergence}
Finally, we will give a characterization of convergence of the
adaptive directional subdivision scheme, like usually in terms of a
(restricted joint) spectral radius. In this subsection, the adaptive
directional subdivision scheme both for masks $a_0$ and $a_1$ as well
as for their associated matrix sequences $B_0$ and $B_1$ will come
into play. To distinguish both, for the first, we again employ the
notation $S_\eps$, $\eps \in E$, whereas the second adaptive directional
subdivision scheme will be denoted by $S^B_{\eps}$, $\eps \in E$.

Now, given two matrix masks
$B_\eps$, $\eps \in \{0,1\}$, their \emph{restricted joint spectral
radius} is defined as
\begin{equation*}
  \rho \left( B_0, B_1 \,|\, \nabla \right)
  = \limsup_{n \to \infty} \sup_{\eps \in \{0,1\}^n} \sup_{c \in
    \nabla \ell_\infty} \left\| S^B_{\eps} c \right\|_\infty^{1/n}.
\end{equation*}
The joint spectral radius is called ``restricted'' since the
supremum is not taken over all $2$--vector valued sequences but only
over the proper subset $\nabla \ell_\infty$, see
\cite{CharinaContiSauer04,Sauer06b}. The main result of this
paragraph is now as follows.

\begin{theorem}\label{T:Convergence}
  The adaptive directional subdivision scheme based on the masks $a_\eps$, $\eps \in
  \{0,1\}$ converges if and only if $a_\eps^* (z) \in I$ and the
  representation masks $B_\eps$, $\eps \in \{0,1\}$, satisfy
  $\rho \left( B_0, B_1 \,|\, \nabla \right) < 1$.
\end{theorem}

\noindent We will split the lengthy proof of
Theorem~\ref{T:Convergence} into several partial results, beginning
with the sufficiency of the spectral radius condition. To that end,
we will show that, starting with a particular test function $g$, the
sequence $g * S_{P_n \eps} c$ converges to a limit function for any
choice of $\eps \in E_\infty$ and any $c \in \ell_\infty$. Indeed,
we choose the test function $g$ to be $W_0$--refinable with respect to a
mask $b$, that is
\begin{equation}\label{eq:choicetest}
g = \sum_{\alpha \in \Z^2} b(\alpha) \, g \left( W_0 \cdot - \alpha
\right).
\end{equation}
Such functions can be easily shown to exist, even with an arbitrary
order of smoothness: pick any cardinal B--spline $\phi = M \left(
\cdot \,|\, 0,\dots,N \right)$ with refinement mask $h$, then a
double application of the refinement equation with respect to the
first variable shows that the tensor product function
\[
g(x,y) = \left[ \left( \phi * \phi \right) \otimes \phi \right]
(x,y) = \left( \phi * \phi \right) (x) \, \phi(y) =: \psi(x) \,
\phi(y)
\]
is $W_0$--refinable with respect to the mask $b = S_{h,2} h \otimes h$,
where $S_{h,2}$ denotes the subdivision scheme with mask $h$ and dilation 2.
The following lemma states a more general process.

\begin{lemma}
\label{lemma:maskgeneration1}
Let $b_1, b_2 \in \ell(\ZZ)$ be $2$-refinable masks, and let the mask $\tilde{b}_1$ be defined by
$\tilde{b}_1(m) =  S_{b_1,2} b_1(m)= \sum_{k \in \ZZ} b_1(k) b_1(m-2k)$. Then the mask
$a_0=\tilde{b}_1 \otimes b_2$ is $W_0$-refinable, and $a_1 = a_0(U
\cdot)$ is $W_1$-refinable.
\end{lemma}

\begin{proof}
Let $\varphi_1, \varphi_2$ be univariate functions which are
$2$-refinable with respect to $b_1, b_2$, respectively, i.e.,
\[\varphi_i = \sum_{k \in \Z} b_i(k) \varphi(2 \cdot - k), \quad i=1, 2.\]
We claim that the function $f$ defined by
\[ f = \varphi_1 \otimes \varphi_2\]
is $W_0$-refinable with respect to $a_0$. Indeed, for $x = (x_1,x_2)
\in \RR^2$, we obtain
{\allowdisplaybreaks
\begin{eqnarray*}
\lefteqn{\sum_{\alpha \in \Z^2} (\tilde{b}_1 \otimes b_2)(\alpha) f(W_0 x - \alpha)}\\
& = & \left[\sum_{\alpha_1 \in \Z} \left(\sum_{k \in \ZZ}
b_1(\alpha_1-2k)b_1(k)\right)
\varphi_1(4 x_1 - \alpha_1)\right] \left[\sum_{\alpha_2 \in \Z} b_2(\alpha_2) \varphi_2(2 x_2 - \alpha_2)\right]\\
& = &  \left[\sum_{k \in \Z} b_1(k) \sum_{\alpha_1 \in \ZZ}
b_1(\alpha_1)
\varphi_1(4 x_1 - 2 k -\alpha_1)\right] \varphi_2(x_2)\\
& = &  \left[\sum_{k \in \Z} b_1(k) \varphi_1(2 x_1 - k)\right] \varphi_2(x_2)\\
& = &  f(x).
\end{eqnarray*}
}
The claim concerning $W_1$-refinability of $a_1$ follows from Lemma
\ref{lemma:relationrefinability}.
\end{proof}

There also exists a canonical $W_1$--refinable function associated
to $g$.

\begin{lemma}\label{lemma:relationrefinability}
  If $g_0 = g$ is $W_0$--refinable with respect to the mask $b_0 = b$, then
  $g_1 = g_0 \left( U \cdot \right)$ is $W_1$--refinable with respect to the mask
  $b_1 = b_0 \left( U \cdot \right)$.
\end{lemma}

\begin{proof}
  Setting $g_1 = g_0 \left( U \cdot \right)$ and thus $g_0 = g_1 \left(
    U^{-1} \cdot \right)$, we find for $x \in \R^2$ that
  \begin{eqnarray*}
    g_1 (x) & = & g_0 (Ux)
    = \sum_{\alpha \in \Z^2} b_0 (\alpha) \, g_0 \left( W_0 U x - \alpha
    \right)
    = \sum_{\alpha \in \Z^2} b_0 (\alpha) \, g_0 \left( W_0 V^2 x - \alpha
    \right) \\
    & = & \sum_{\alpha \in \Z^2} b_0 (\alpha) \, g_1 \left( U^{-1} W_1 V x -
      U^{-1} \alpha \right)
    = \sum_{\alpha \in \Z^2} b_0 \left( U \alpha \right) \,
    g_1 \left( W_1 x - \alpha \right) \\
    & = & \sum_{\alpha \in \Z^2} b_1 (\alpha) \, g_1 \left( W_1 x - \alpha
    \right),
  \end{eqnarray*}
  hence $g_1$ is $W_1$--refinable with respect to $b_1$.
\end{proof}

\noindent The next two observation are again of a more algebraic
nature.

\begin{lemma}\label{L:NullFactMask}
  Suppose that a mask $a$ satisfies $S_{a,W_\eps} c = 0$ for all
  \emph{constant} sequences $c$ and some $\eps \in \{0,1\}$. Then
  there exists a $1 \times 2$ matrix mask $B$ such that
  $S_{a,W_\eps} = S_{B,W_\eps} \nabla$.
\end{lemma}

\begin{proof}
  Again we refer to \cite{MoellerSauer04,Sauer02} where it has been shown
  that $S_{a,W_\eps} c = 0$ for all constant sequences $c$ if and only if
  $a^* (z) \in \left\langle z^{W_\eps} - 1 \right\rangle$
  which is in turn equivalent to the existence of
  a representation
  \[
  a^* (z) = b_1^* (z) \left( z^{\left( W_\eps \right)_1} - 1 \right)
  +  b_2^* (z) \left( z^{\left( W_\eps \right)_2} - 1 \right)
  = B^* (z) \, \left[ z^{W_\eps} - 1 \right],
  \]
  which is nothing but $S_{a,W_\eps} = S_{B,W_\eps} \nabla$.
\end{proof}

\begin{lemma}\label{L:NullFactFunc}
  Suppose that a compactly supported function $f$ satisfies $f * c = 0$ for
  all constant sequences, then there exists a compactly supported, continuous
  $1 \times 2$ matrix function $G$ such that $f * c = G * \nabla c$ for
  all $c \in \ell_\infty \left( \Z^2 \right)$.
\end{lemma}

\begin{proof}
  For any $x \in [0,1]^2$ we consider the sequence $f_x = \left( f(x+\alpha)
    \;:\; \alpha \in \Z^2 \right)$. Since $f$ is compactly supported, any
  such sequence $f_x$, $x \in [0,1]^2$
  has finite support and since $f$ is continuous, the map
  $x \mapsto f_x$ is a continuous one.

  By assumption, $f_x * c = 0$ for any $x$ and any constant sequence $c$,
  hence, with the scaling matrix $I$, the same methods as above yield
  that $f_x^* \in \left\langle z^I - 1 \right\rangle = \left\langle z - 1
  \right\rangle$. Consequently, we have that
  \[
  f_x^* (z) = g_{x,1}^* (z) \, \left( z_1 - 1 \right) +  g_{x,2}^* (z) \,
  \left( z_2 - 1 \right)
  = G_x^* (z) \, \left[ z - 1 \right]
  \]
  where, like $f_x$ and $f_x^* (z)$, also $G_x^* (z)$ depend continuously
  on $x$ as they can be obtained by applying the orthogonal reduction process
  from \cite{Sauer01}. Therefore, the function $G$, defined as
  \[
  G (x + \alpha) = G_x (\alpha), \qquad x \in [0,1]^2, \quad \alpha \in \Z^2
  \]
  has the properties claimed in the statement of the lemma.
\end{proof}

\noindent Now we are in position to prove the sufficiency of the
spectral radius condition which we state as a separate proposition.

\begin{proposition}\label{P:Convergence<=}
  The adaptive directional subdivision scheme based on the masks $a_\eps$, $\eps \in
  \{0,1\}$ converges, if $a_\eps^* (z) \in I$ and the
  representation masks $B_\eps$, $\eps \in \{0,1\}$, satisfy
  $\rho \left( B_0, B_1 \,|\, \nabla \right) < 1$.
\end{proposition}

\begin{proof}
  For any $\theta \in \left( 0,1-\rho \right)$,
 there exists, by standard
  properties of the (joint) spectral radius, a constant $C > 0$ such that
  \begin{equation*}
    \left\| S^B_{P_n \eps} \nabla c \right\|_\infty \le
    C \left( \rho + \theta \right)^n = C \sigma^n,
    \qquad n \in \N, \quad \eps \in
    E_\infty, \quad c \in \ell_\infty \left( \ZZ^2 \right),
  \end{equation*}
  where $0 < \sigma := \rho + \theta < 1$.

  Now, let $\eps \in E_\infty$ be given and suppose first that
  $\eps_n = 0$. Then, by the refinability of the test function $g$ from \eqref{eq:choicetest} and
  Lemma~\ref{L:NullFactMask} which ensures the existence of a finitely
  supported matrix mask $F$ such that $S_0 - S_{b,W_0} = S_{F,W_0} \nabla$, we have
  that
  \begin{eqnarray*}
    \lefteqn{
      \left\| g * S_{P_n \eps} c \left( W_{P_n \eps} \cdot \right) -
        g * S_{P_{n-1} \eps} c \left( W_{P_{n-1} \eps} \cdot \right)
      \right\|_\infty } \\
    & = & \hspace*{-0.2cm} \left\| g * S_{P_n \eps} c \left( W_{P_n \eps} \cdot
      \right) - g * S_{b,W_0} S_{P_{n-1} \eps} c \left( W_{P_{n} \eps}
        \cdot \right) \right\|_\infty \\
    & = & \left\| g *  \left( S_0 - S_{b,W_0} \right) S_{P_{n-1} \eps} c
    \right\|_\infty \\
    & \le & \hspace*{-0.2cm} B_g \, \left\| \left( S_0 - S_{b,W_0} \right)
      S_{P_{n-1} \eps} c \right\|_\infty
    =  B_g \, \left\| S_{F,W_0} \nabla S_{P_{n-1} \eps} c \right\|_\infty \\
    & = & B_g \, \left\| S_{F,W_0} S_{P_{n-1} \eps}^B \nabla c \right\|_\infty
    \le B_g \, \left\| S_{F,W_0} \right\| \, C \, \sigma^{n-1}.
  \end{eqnarray*}
  If on the other hand $\eps_n = 1$, by using the function $g_1 = g(U\,\cdot)$
  (cf. Lemma \ref{lemma:relationrefinability}) we pass to the estimate
  \begin{eqnarray*}
    \lefteqn{
      \left\| g * S_{P_n \eps} c \left( W_{P_n \eps} \cdot \right) -
        g * S_{P_{n-1} \eps} c \left( W_{P_{n-1} \eps} \cdot \right)
      \right\|_\infty } \\
    & \le & \left\| \left( g - g_1 \right) * S_{P_n \eps} c
      \left( W_{P_n \eps} \cdot \right) \right\|_\infty
    + \left\| \left( g - g_1 \right) * S_{P_{n-1} \eps} c
      \left( W_{P_{n-1} \eps} \cdot \right) \right\|_\infty \\
    & & \qquad + \left\| g_1 * S_{P_n \eps} c \left( W_{P_n \eps}
        \cdot \right)
      - g_1 * S_{P_{n-1} \eps} c \left( W_{P_{n-1} \eps}
        \cdot \right) \right\|_\infty.
  \end{eqnarray*}
  For the first two terms we now make use of Lemma~\ref{L:NullFactFunc}
  to obtain that
  \begin{eqnarray*}
    \lefteqn{
      \left\| \left( g - g_1 \right) * S_{P_n \eps} c
        \left( W_{P_n \eps} \cdot \right) \right\|_\infty } \\
    & = & \left\| \left( g - g_1 \right) * S_{P_n \eps} c \right\|_\infty
    = \left\| G * \nabla S_{P_n \eps} c \right\|_\infty
    = \left\| G * S_{P_n \eps}^B \nabla c \right\|_\infty
    \le B_G \, C \, \sigma^n
  \end{eqnarray*}
  and
  \[
  \left\| \left( g - g_1 \right) * S_{P_{n-1} \eps} c
    \left( W_{P_{n-1} \eps} \cdot \right) \right\|_\infty \le B_G \, C \,
  \sigma^{n-1},
  \]
  respectively, while the third term can now be estimated as above again.
  In summary, we obtain that there exists a constant $D > 0$ such that
  \[
  \left\| g * S_{P_n \eps} c \left( W_{P_n \eps} \cdot \right) -
    g * S_{P_{n-1} \eps} c \left( W_{P_{n-1} \eps} \cdot \right)
  \right\|_\infty \le D \, \sigma^{n-1}
  \]
  so that for $m \in \N$
  \[
  \left\| g * S_{P_{n+m} \eps} c \left( W_{P_{n+m} \eps} \cdot
    \right) - g * S_{P_n \eps} c \left( W_{P_n \eps} \cdot \right)
  \right\|_\infty \le D \, \frac{\sigma^n}{1-\sigma}.
  \]
  In other words, the sequence $g * S_{P_n \eps} c \left( W_{P_n
      \eps} \cdot \right)$ is a Cauchy sequence of continuous functions
  and thus must converge to a limit function for $n \to \infty$.
  Convergence of the subdivision scheme then follows by standard means.
\end{proof}

\noindent The proof of the converse statement of
Proposition~\ref{P:Convergence<=} is based on the estimate
\begin{eqnarray*}
  \left\| S^B_{P_n \eps} \, \nabla \delta \right\|_\infty & = &
  \left\| \nabla S_{P_n \eps} \, \delta \right\|_\infty = \left\|
    \begin{pmatrix}
      S_{P_n \eps} \delta \left( \cdot - \eta_1 \right) -
      S_{P_n \eps} \delta \left( \cdot \right) \\
      S_{P_n \eps} \delta \left( \cdot - \eta_2 \right) -
      S_{P_n \eps} \delta \left( \cdot \right)
    \end{pmatrix}
  \right\|_\infty \\
  & = & \max_{j=1,2} \left\| S_{P_n \eps} \delta
    \left( \cdot - \eta_j
    \right) - S_{P_n \eps} \delta \left( \cdot \right) \right\|_\infty \\
  & \le & \max_{j=1,2} \left(
    \left\| S_{P_n \eps} \delta \left( \cdot - \eta_j \right)
      - \sigma_{W_{P_n \eps}^{-1}} f \left( \cdot - \eta_j \right)
    \right\|_\infty \hspace*{-0.2cm}
    + \left\| S_{P_n \eps} \delta - \sigma_{W_{P_n \eps}^{-1}} f
    \right\|_\infty \right. \\
  & & \qquad \left.
    + \left\| \sigma_{W_{P_n \eps}^{-1}} f \left( \cdot - \eta_j
      \right)
      - \sigma_{W_{P_n \eps}^{-1}} f \left( \cdot \right)
    \right\|_\infty
  \right),
\end{eqnarray*}
hence,
\begin{equation*}
  \left\| S^B_{P_n \eps} \, \nabla \delta \right\|_\infty \le \max_{j=1,2} \left(
    2 \left\| S_{P_n \eps} \delta - \sigma_{W_{P_n \eps}^{-1}} f
    \right\|_\infty
    + \left\| \sigma_{W_{P_n \eps}^{-1}} f \left( \cdot - \eta_j
      \right)
      - \sigma_{W_{P_n \eps}^{-1}} f \left( \cdot \right)
    \right\|_\infty
  \right).
\end{equation*}
If we assume that the subdivision scheme converges with uniformly
continuous limit function, then the right hand side converges to
zero, hence also $\left\| S^B_{P_n \eps} \, \nabla c \right\|_\infty \to 0$
for $n \to \infty$ and any $c \in \ell_\infty \left( \ZZ^2 \right)$.
This, however, is not sufficient for our purposes. To show that the
restricted spectral radius of $\rho \left( B_0,B_1 \,|\, \nabla
\right)$ is less than one, we have to show that
\begin{equation}
  \label{eq:S_BEstToShow}
  \left\| S^B_{P_n \eps} \, \nabla c \right\|_\infty \le C \,
  \theta_n \, \left\|
    \nabla c \right\|_\infty, \qquad \lim_{n \to \infty} \theta_n = 0,
\end{equation}
which will be prepared in the next lemmas. Here we follow the
outline of a proof from \cite{CavarettaDahmenMicchelli91} and show
that there exists a constant $C > 0$ such that
\[
\left\| \nabla S_{P_n \eps} c \right\|_\infty \le C \, \theta_n \,
\left\| \nabla c \right\|_\infty, \qquad \lim_{n \to \infty} \theta_n = 0,
\]
from which (\ref{eq:S_BEstToShow}) follows immediately. We begin
with an estimate on the limit function $f_\eps$.

\begin{lemma}\label{L:Conv=>SR1}
  If $a_0$ and $a_1$ define a convergent subdivision scheme,
  then there exists a constant $C_1 > 0$
  such that for any $\eps \in E_\infty$ and any
  $c \in \ell_\infty \left( \ZZ^2 \right)$
  \[
  \left| f_\eps * c (x) - f_\eps * c (y) \right| \le
  C_1 \, \omega \left( f_\eps, \delta \right) \,
  \left\| \nabla c \right\|_\infty, \qquad
  \left\| x - y \right\| \le \delta < 1.
  \]
\end{lemma}

\begin{proof}
  Since, according to Lemma~\ref{lemma:subdivisionconvergent},
  convergence implies the preservation of constant sequences by the
   subdivision scheme, we also have that
  \[
  1 = f_\eps * 1 = \sum_{\alpha \in \Z^2} f_\eps (\cdot - \alpha)
  \]
  and thus, for any $c \in \ell_\infty$, any $w \in \R$ and any
  $x,y \in \R^2$ with $\left\| x - y \right\| \le \delta$,
  \begin{eqnarray*}
    \left| f_\eps * c (x) - f_\eps * c (y) \right|
    & = & \left| \sum_{\alpha \in \Z^2} \left( f_\eps (x-\alpha) -
        f_\eps (y-\alpha) \right) \left( c(\alpha) - w \right) \right| \\
    & \le & \# \Omega_{x,y} \,\cdot\, \omega \left( f_\eps, \delta \right) \,\cdot\,
  \max_{\alpha \in \Omega_{x,y}} \left| c(\alpha) - w \right|,
  \end{eqnarray*}
  where
  \[
  \Omega_{x,y} = \left\{ \alpha \in \Z^2 \;:\; f_\eps (x-\alpha) \neq 0
  \right\} \cup \left\{ \alpha \in \Z^2 \;:\; f_\eps (y-\alpha) \neq 0
  \right\}.
  \]
  Since $f_\eps$ is finitely supported, we have that $\# \Omega_{x,y}
  < \infty$. Specifically, if we assume that $f_\eps$ is supported on
  $[-N,N]^2$, then $\# \Omega_{x,y} \le \left( 2N+2 \right)^2$ as long as
  $\delta < 1$. Choosing
  \[
  w = \frac12 \left( \max_{\alpha \in \Omega_{x,y}} c(\alpha) +
    \min_{\alpha \in \Omega_{x,y}} c(\alpha) \right),
  \]
  it follows for any $\alpha \in \Omega_{x,y}$ that
  \[
  \left| c(\alpha) - w \right| \le
  \frac12 \left| \max_{\alpha \in \Omega_{x,y}} c(\alpha) +
    \min_{\alpha \in \Omega_{x,y}} c(\alpha) \right|
  \le \frac12 \, \# \Omega_{x,y} \left\| \nabla c \right\|_\infty,
  \]
  hence,
  \[
  \left| f_\eps * c \, (x) - f_\eps * c \, (y) \right| \le
  \frac12 \left( 2N+2 \right)^4 \, \omega \left( f_\eps, \delta \right) \,
  \left\| \nabla c \right\|_\infty
  \]
  as claimed.
\end{proof}

\noindent The next result concerns the difference between the
subdivision scheme and the limit function.

\begin{lemma}\label{L:Conv=>SR2}
  If the adaptive directional subdivision scheme based on the masks $a_\eps$, $\eps \in \{0,1\}$,
  then there exists a constant $C_2 > 0$ such that,  for any $n \in \N$, we have
  \[
  \left\| S_{P_n \eps} c - f_\eps * c \left(
      W_{P_n \eps}^{-1}
      \cdot \right) \right\|_\infty \le C_2 \,
  \left\| S_{P_n \eps}
    \delta - f_\eps \left( W_{P_n \eps}^{-1} \cdot \right)
  \right\|_\infty \, \left\| \nabla c \right\|_\infty.
  \]
\end{lemma}

\begin{proof}
  We fix $n$, set, for abbreviation, $\widehat \eps = P_n \eps$,
  and assume again that $f_\eps$ as well as $a_0$ and $a_1$ are
  supported on $[-N,N]^2$. Again, we make use of the fact that $S_{\widehat
    \eps}$ and $f_\eps$ preserve constant data and obtain, for
  any $\alpha \in \Z^2$ and $w \in \R$, that
  \[
  S_{\widehat \eps} c (\alpha) - f_\eps * c \left(
    W_{\widehat \eps}^{-1} \alpha \right)
  = \sum_{\beta \in \Z^2} \left( a_{\widehat \eps}
    \left( \alpha - W_{\widehat \eps} \beta \right) - f_\eps
    \left( W_{\widehat \eps}^{-1} \alpha - \beta \right) \right)
  \left( c(\beta) - w \right).
  \]
  Since
  \[
  \Omega_{\alpha,\widehat \eps} = \left\{ \alpha \in \Z^2 \;:\;
    a_{\widehat \eps} \left( \alpha - W_{\widehat \eps} \beta \right)
    \neq 0  \right\} \cup  \left\{ \alpha \in \Z^2 \;:\;
    f_\eps \left( W_{\widehat \eps}^{-1} \alpha - \beta \right)
    \neq 0  \right\}
  \]
  again satisfies $\# \Omega_{\alpha,\widehat \eps} \le \left( 2N+2
  \right)^2$, the same judicious choice of $w$ as above leads to the
  estimate
  \[
  \left| S_{\widehat \eps} c (\alpha) - f_\eps * c \left(
      W_{\widehat \eps}^{-1} \alpha \right) \right|
  \le \left( 2N+2 \right)^4 \, \sup_{\alpha \in \Z^2} \left|
    a_{\widehat \eps} \left( \alpha \right) - f_\eps
    \left( W_{\widehat \eps}^{-1} \alpha \right) \right| \,
  \left\| \nabla c \right\|_\infty,
  \]
  from which the claim follows immediately.
\end{proof}

\noindent Now it is easy to complete the proof of the converse
statement for convergence which we formulate in the following way.

\begin{proposition}
  If the adaptive directional subdivision scheme based on the masks $a_\eps$, $\eps \in
  \{0,1\}$ converges then $a_\eps^* (z) \in I$ and the
  representation masks $B_\eps$, $\eps \in \{0,1\}$, satisfy
  $\rho \left( B_0, B_1 \,|\, \nabla \right) < 1$.
\end{proposition}

\begin{proof}
  In Lemma~\ref{lemma:subdivisionconvergent}, it has already been shown that
  convergence implies $a_\eps^* (z) \in I$. Moreover,
  Lemma~\ref{L:Conv=>SR1} and Lemma~\ref{L:Conv=>SR2} allow us to conclude
  with $C = \max \{C_1,C_2\}$
  that, for any $c \in \ell_\infty$, we have
  \begin{eqnarray*}
    \left\| S^B_{P_n \eps} \nabla c \right\|_\infty & = &
    \left\| \nabla S_{P_n \eps} c \right\|_\infty\\
    & \le &  \left\| \nabla \left( S_{P_n} c - f_\eps * c \left(
          W_{P_n \eps}^{-1} \cdot \right) \right) \right\|_\infty +
    \left\| \nabla f_\eps * c \left(
        W_{P_n \eps}^{-1} \cdot \right) \right\|_\infty \\
    & \le & 2 C \left( \left\| S_{P_n \eps} \delta - f_\eps
        \left( W_{P_n \eps}^{-1} \cdot \right)
      \right\|_\infty + \omega \left( f_\eps,
        \left\| W_{P_n \eps}^{-1} \right\| \right) \right)
    \left\| \nabla c \right\|_\infty,
  \end{eqnarray*}
  and since
  \[
  \lim_{n \to \infty} \left\| S_{P_n \eps} \delta - f_\eps
    \left( W_{P_n \eps}^{-1} \cdot \right)
  \right\|_\infty + \omega \left( f_\eps,
    \left\| W_{P_n \eps}^{-1} \right\| \right) = 0
  \]
  by convergence of the adaptive directional subdivision scheme and uniform continuity of the
  limit function, our prove is complete.
\end{proof}


\section{Numerical Experiments}
\label{sec:numerics}

In this section we present some numerical experiments which
illustrate the ability of the developed class of subdivision schemes
to adaptively change the orientation of the data.

First, we recall that there exist a general way to construct masks,
which are refinable with respect to the dilation matrices $W_0$ and
$W_1$, compare Lemma \ref{lemma:maskgeneration1}.
Now let the mask $b \in \ell(\ZZ)$ be chosen by $b(-3) =
-\frac{1}{16} = b(3)$, $b(-1)= \frac{9}{16} = b(1)$, $b(0)=1$ and
$b(m) = 0$ otherwise, which coincides with the mask studied by
Deslauriers and Dubuc \cite{DD89}. We remark that this mask yields
a 2-interpolatory subdivision scheme (compare also Section
\ref{sec:shearletMRA}). By Lemma \ref{lemma:maskgeneration1}, we know that $a_0 =
\tilde{b} \otimes b$ is $W_0$-refinable, and $a_1 = a_0(U \cdot)$ is
$W_1$-refinable.

In Figure \ref{fig:line} we illustrate the refinement of the matrix
\begin{equation} \label{eq:C1}
C_1 = \begin{pmatrix}
     0 & 1 & 0\\
     0 & 1 & 0\\
     0 & 1 & 0
 \end{pmatrix},
\end{equation}
\begin{figure}[h]
\begin{center}
\begin{picture}(300,250)(0,0)
\put(-5,140)%
{\ifpdf
\includegraphics[width=5cm]{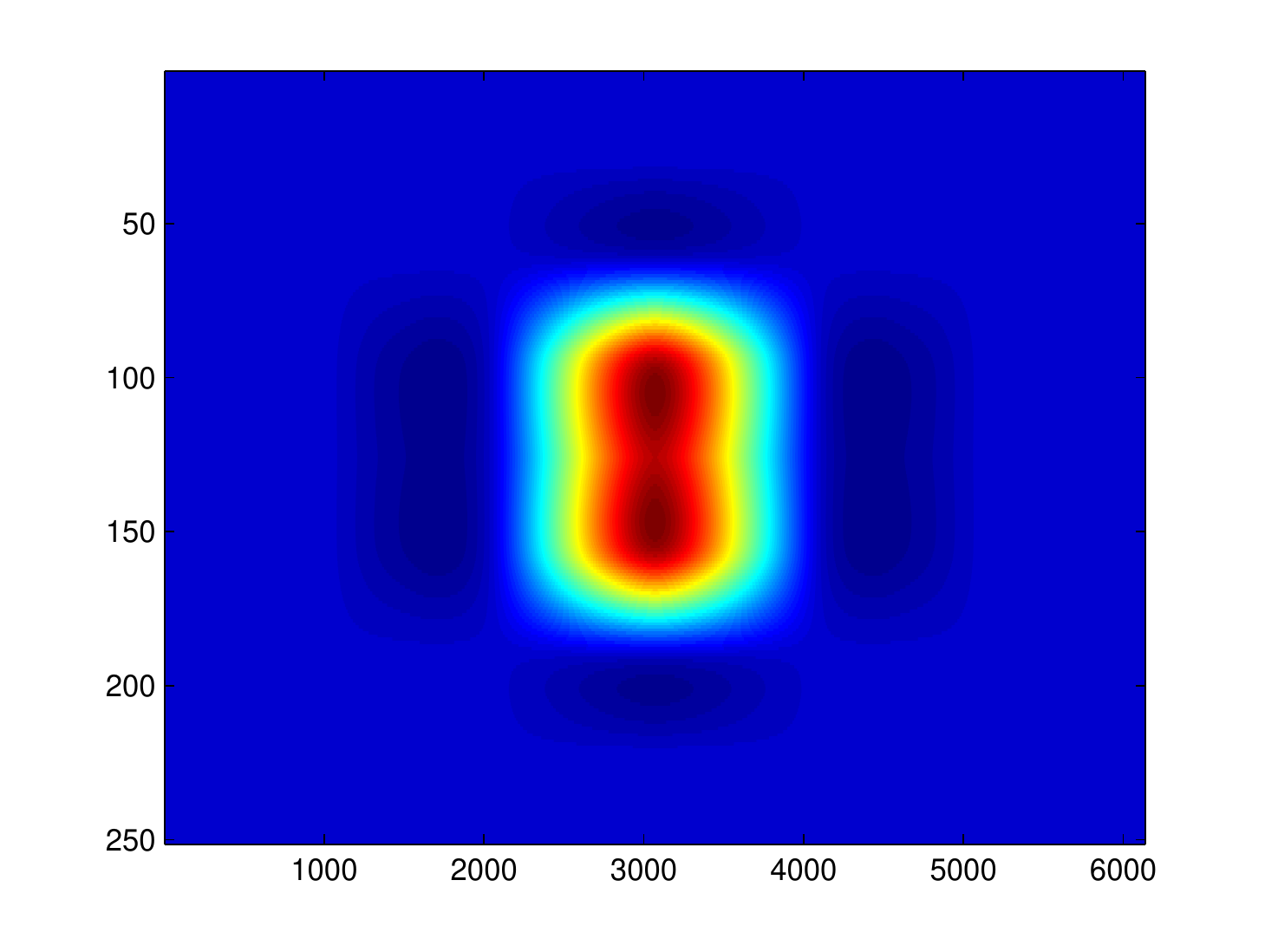}
\else
\includegraphics[width=5cm]{line1.eps}
\fi}
  \put(60,130){(a)}
  \put(160,140)%
  {\ifpdf
    \includegraphics[width=5cm]{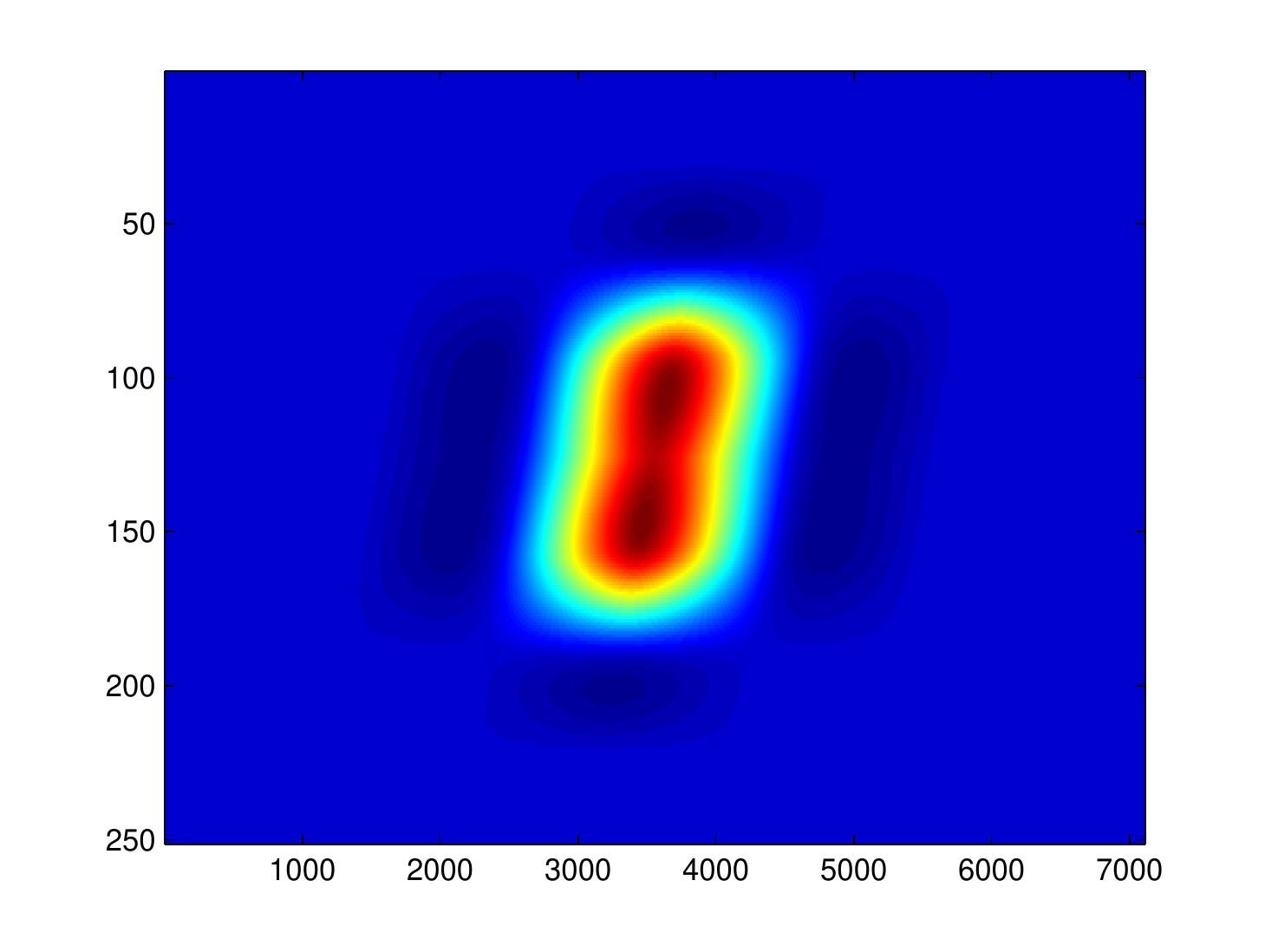}
    \else
    \includegraphics[width=5cm]{line2.eps}
    \fi}
  \put(225,130){(b)}
  \put(-5,10)%
  {\ifpdf
    \includegraphics[width=5cm]{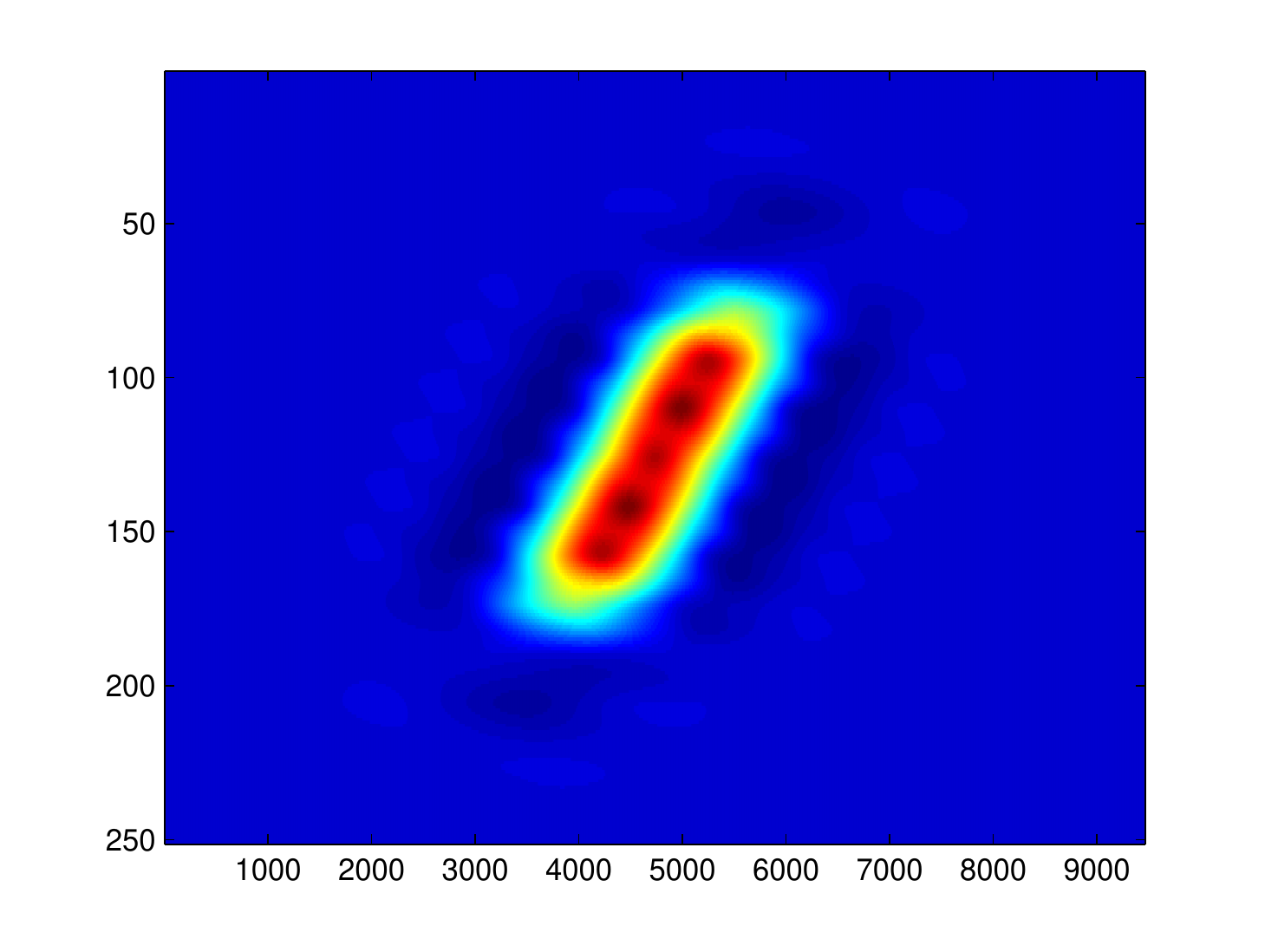}
    \else
    \includegraphics[width=5cm]{line3.eps}
  \fi}
  \put(60,0){(c)}
  \put(160,10)%
  {\ifpdf
    \includegraphics[width=5cm]{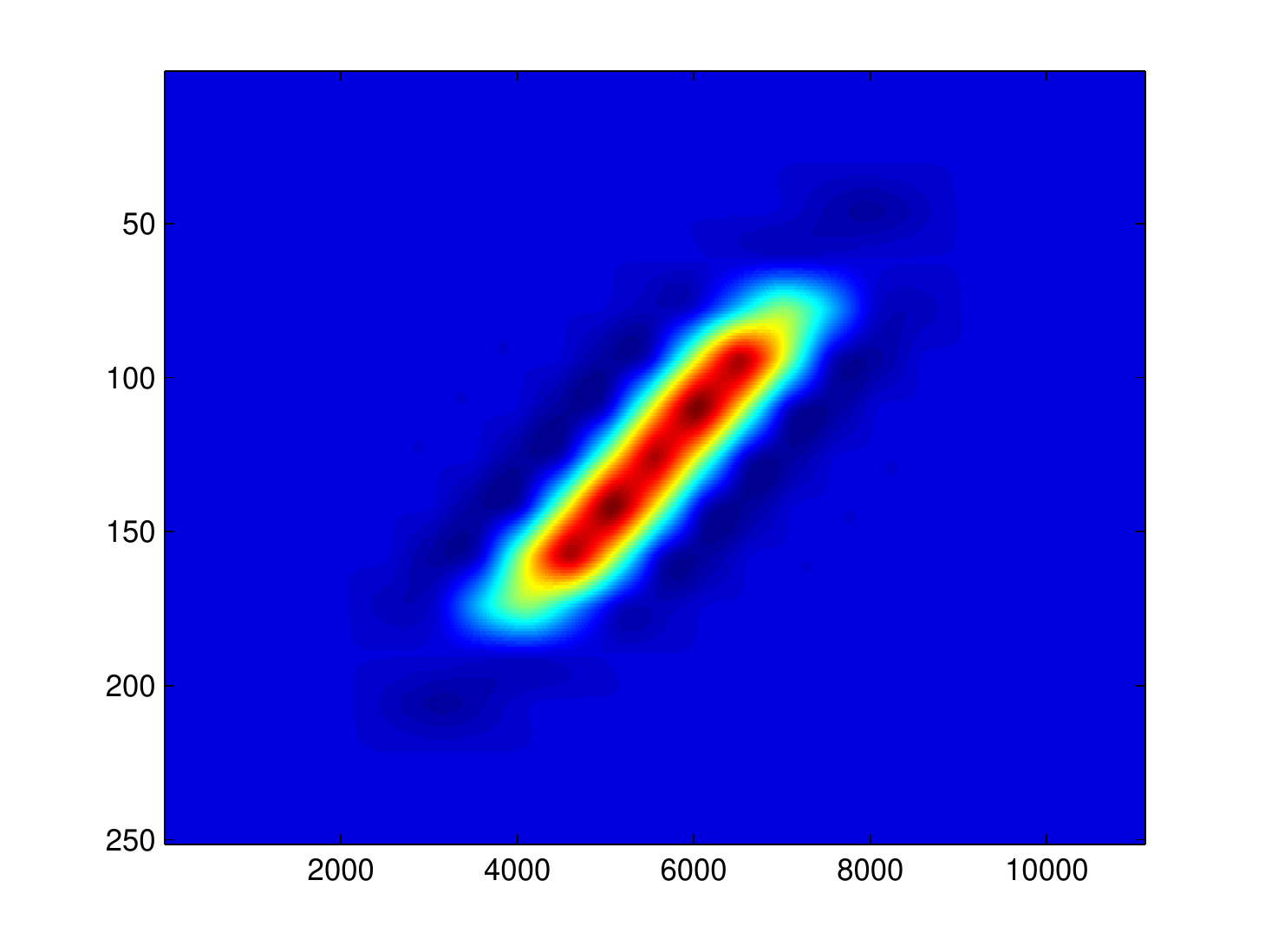}
    \else
    \includegraphics[width=5cm]{line4.eps}
    \fi}
  \put(225,0){(d)}
  \end{picture}
\end{center}
\caption{This figure shows the refinement of the matrix $C_1$
defined in \eqref{eq:C1} after applying $S_\eps$ with (a) $\eps =
(0,0,0,0,0)$,  (b) $\eps = (0,0,0,1,0)$, (c) $\eps = (0,1,0,0,0)$,
and (d) $\eps = (0,1,1,1,1)$.} \label{fig:line}
\end{figure}

and in Figure \ref{fig:cross} we subdivide the data given by
\begin{equation} \label{eq:C2}
C_2 = \begin{pmatrix}
     0 & \frac12 & 0 & 0 & 0\\
     0 & \frac12 & 0 & 0 & 0\\
     1 & 1 & 1 & 1 & 1\\
     0 & 0 & 0 & \frac12 & 0\\
     0 & 0 & 0 & \frac12 & 0
 \end{pmatrix}.
\end{equation}
In both figures we employ different iterations of the subdivision
schemes $S_0$ and $S_1$. As can clearly be seen, the application of
$S_1$ increases the angle the resulting images is sheared in the
$x$-direction, where the angle depends on the particular path in the
binary tree (see Figure \ref{fig:binary}) we choose.

\begin{figure}[h]
\begin{center}
\begin{picture}(300,250)(0,0)
\put(-5,140)%
{\ifpdf
\includegraphics[width=5cm]{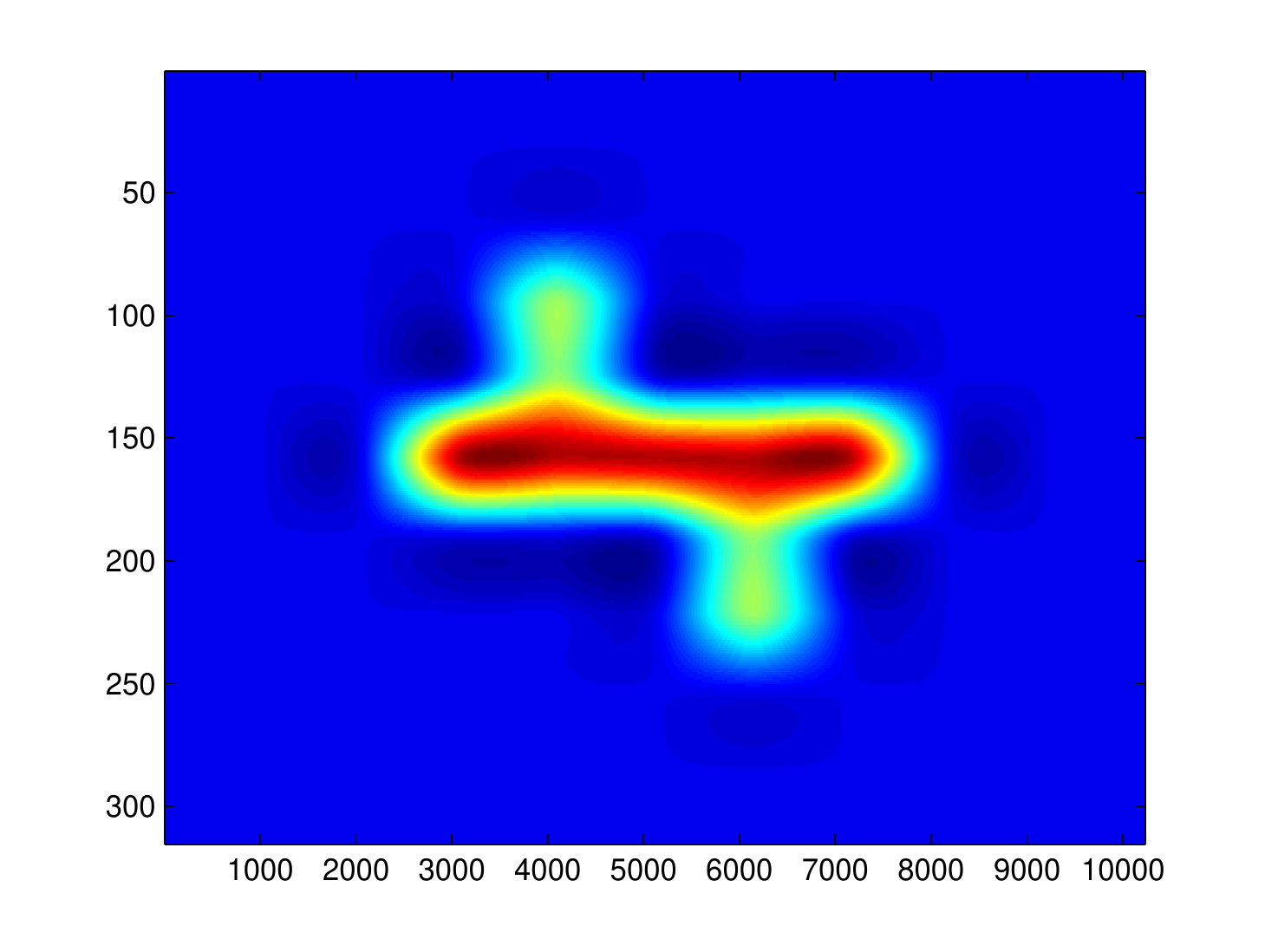}
\else
\includegraphics[width=5cm]{cross1.eps}
\fi}
  \put(60,130){(a)}
  \put(160,140)%
  {\ifpdf
    \includegraphics[width=5cm]{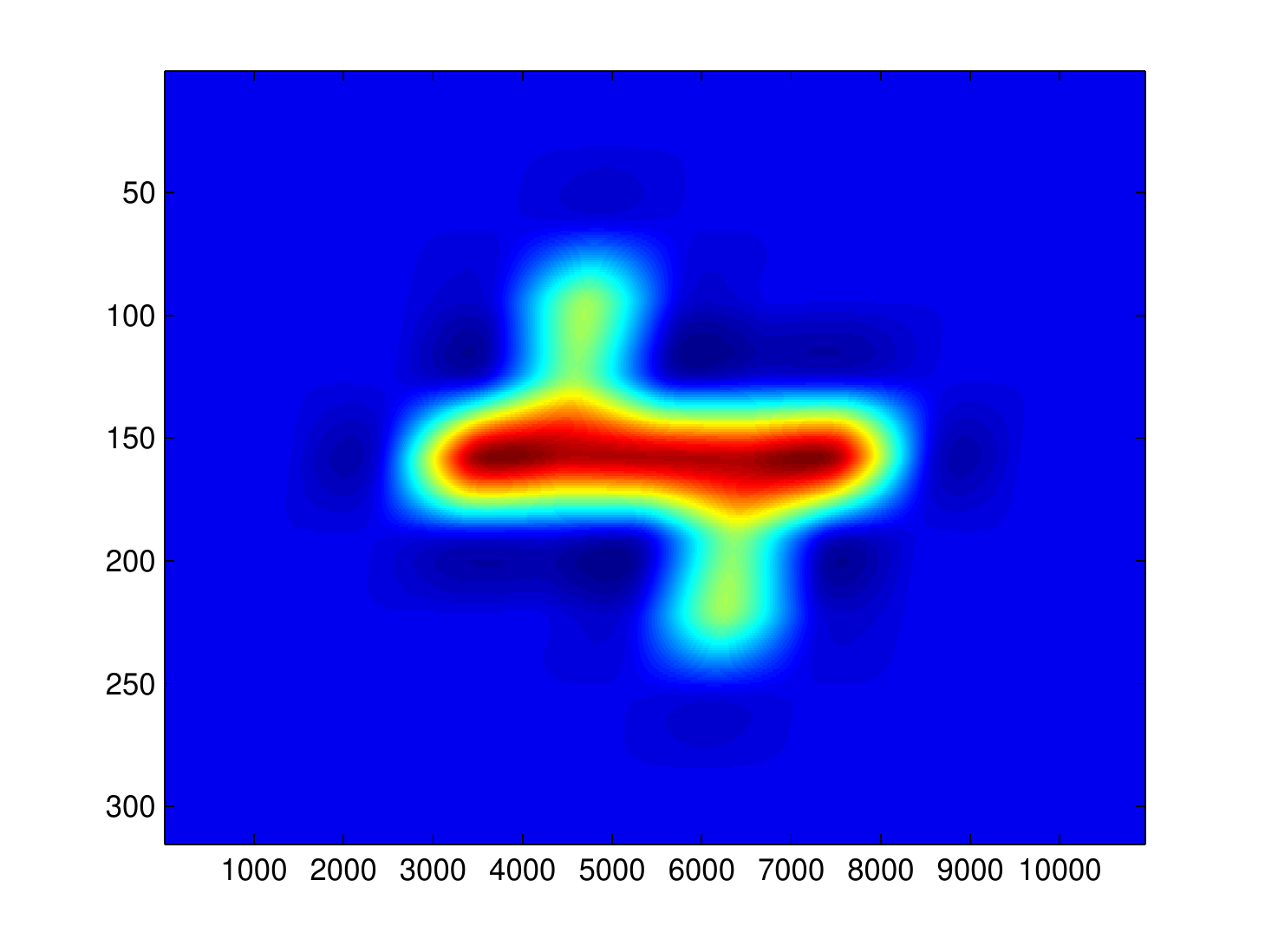}
    \else
    \includegraphics[width=5cm]{cross2.eps}
    \fi}
  \put(225,130){(b)}
  \put(-5,10)%
  {\ifpdf
    \includegraphics[width=5cm]{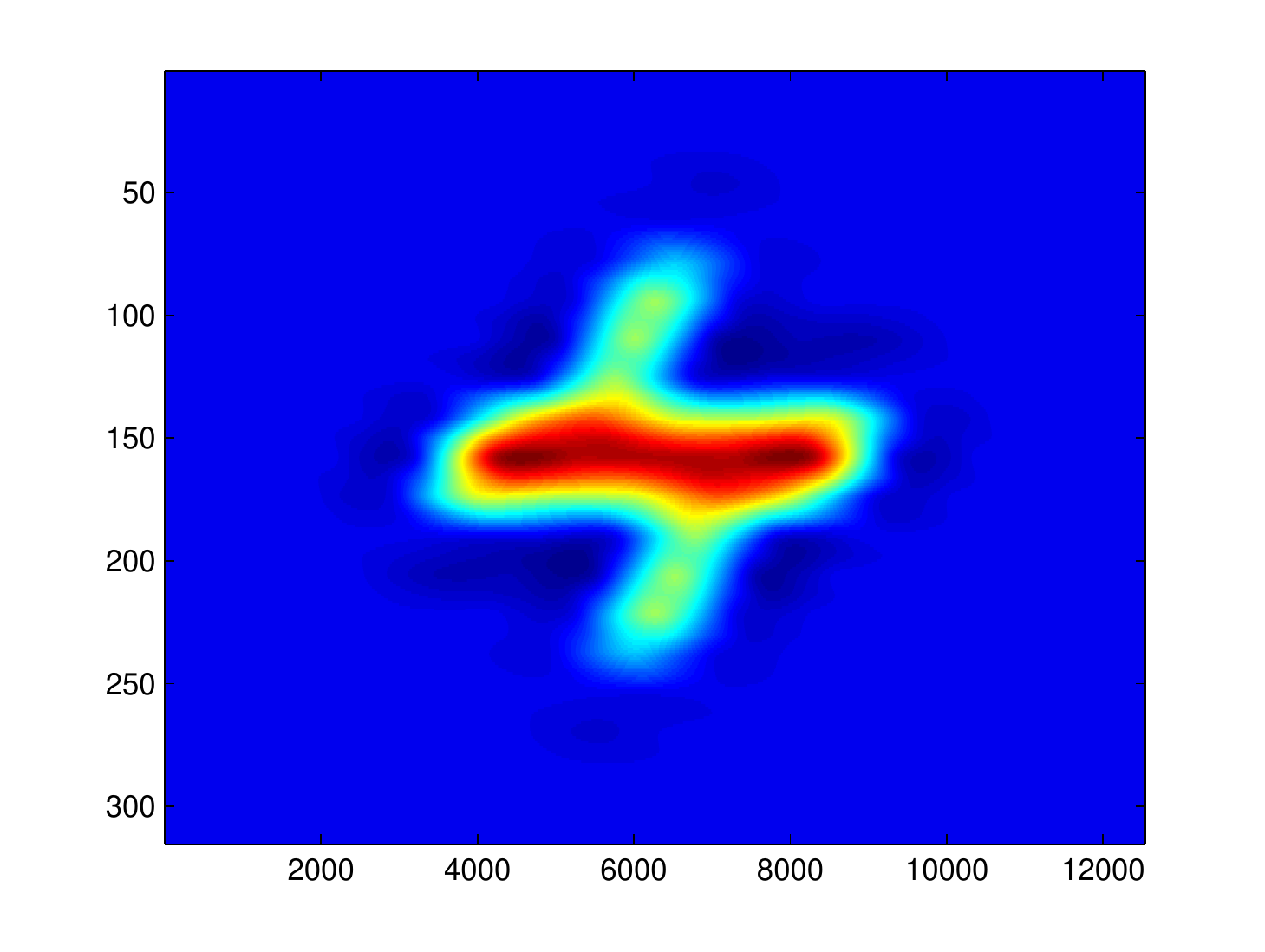}
    \else
    \includegraphics[width=5cm]{cross3.eps}
  \fi}
  \put(60,0){(c)}
  \put(160,10)%
  {\ifpdf
    \includegraphics[width=5cm]{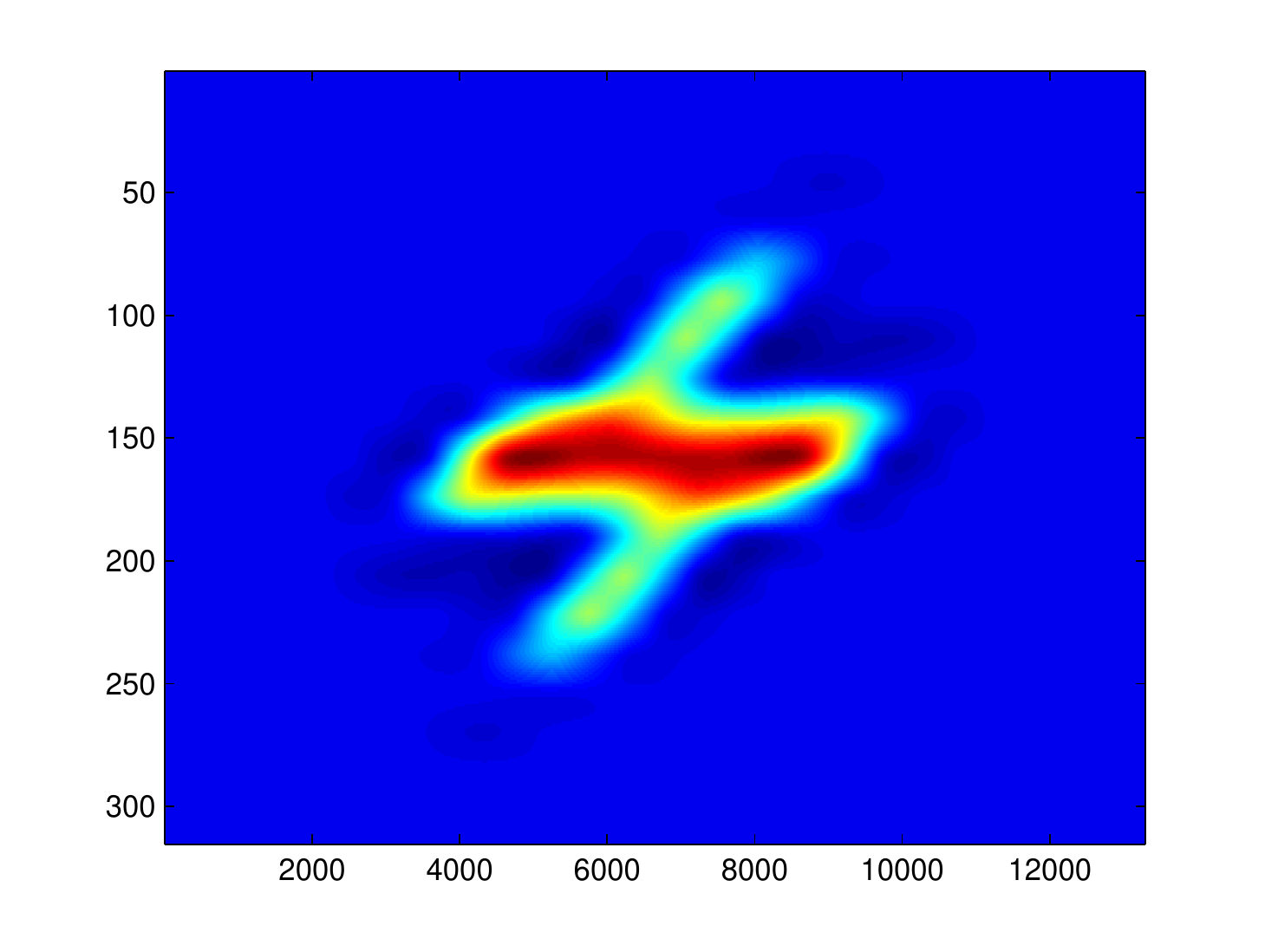}
    \else
    \includegraphics[width=5cm]{cross4.eps}
    \fi}
  \put(225,0){(d)}
  \end{picture}
\end{center}
\caption{This figure shows the refinement of the matrix $C_2$
defined in \eqref{eq:C2} after applying $S_\eps$ with (a) $\eps =
(0,0,0,0,0)$,  (b) $\eps = (0,0,0,1,0)$, (c) $\eps = (0,1,0,0,0)$,
and (d) $\eps = (0,1,1,1,1)$.} \label{fig:cross}
\end{figure}


\section{Shearlet Multiresolution Analysis}
\label{sec:shearletMRA} In this section we will show how the
adaptive directional subdivision schemes developed in the previous
sections can be applied to derive a shearlet multiresolution
analysis. For the sake of simplicity, in the computation of ``dual
functions'' we will restrict ourselves to interpolatory subdivision
schemes in this paper. Our idea is inspired by similar ideas for the
construction of a fast wavelet decomposition from interpolatory
subdivision schemes \cite{Don92}. The construction of a shearlet
multiresolution analysis associated with general adaptive
directional subdivision schemes is beyond the scope of this paper,
and will be studied in a forthcoming paper.


Before constructing the scaling spaces we first need to discuss
whether there exist masks $a_0$ and $a_1$ such that the subdivision
schemes $S_0$ and $S_1$ are both interpolatory, respectively, which
immediately implies that $S_\eps$ is interpolatory for each $\eps
\in E_\infty$. To that end, we proceed by using a tensor product
approach. Recall that a mask $a_0$ leads to an interpolatory
subdivision scheme $S_0$ provided that
\begin{equation}\label{eq:interpolatory0}
a_0(W_0\alpha) = \delta_{\alpha,0}\quad \mbox{for all }\alpha \in
\ZZ^2,
\end{equation}
likewise does a mask $a_1$ lead to an interpolatory subdivision
scheme $S_1$ provided that
\begin{equation}\label{eq:interpolatory1}
a_1(W_1\alpha) = \delta_{\alpha,0}\quad \mbox{for all }\alpha \in
\ZZ^2.
\end{equation}
There exists a canonical way to define $a_1$ by means of the matrix
$U$ as indicated by the following lemma (compare also Lemma
\ref{lemma:maskgeneration1}).

\begin{lemma}
\label{lemma:maskgeneration} Let $b_1, b_2 \in \ell(\ZZ)$ be masks
which satisfy $b_i(2m)=\delta_{m,0}$ for all $m \in \ZZ$, $i=1,2$
and let the mask $\tilde{b}_1$ be defined by $\tilde{b}_1(m) =
S_{b_1,2}b_1 (m)= \sum_{k \in \ZZ} b_1(k) b_1(m-2k)$. Then the mask $\tilde{b}_1
\otimes b_2$ satisfies \eqref{eq:interpolatory0}, and the mask
$(\tilde{b}_1 \otimes b_2)(U \, \cdot)$ satisfies
\eqref{eq:interpolatory1}.
\end{lemma}

\begin{proof}
Given some $\alpha = (\alpha_1,\alpha_2) \in \ZZ^2$, we obtain
\[ (\tilde{b}_1 \otimes b_2)(W_0 \alpha)
= \sum_{k \in \ZZ} b_1(k) b_1(4\alpha_1-2k) \cdot \delta_{2\alpha_2,0}
= \sum_{k \in \ZZ} b_1(k) \delta_{2\alpha_1-k,0} \cdot \delta_{\alpha_2,0}
= \delta_{\alpha,0}.\]
A similar computation shows $(\tilde{b} \otimes b)(U W_1 \alpha) = \delta_{\alpha,0}$.
\end{proof}

\noindent Suppose we have chosen masks $a_0$ and $a_1$ so that the
subdivision scheme $S_\eps$ is interpolatory and converges for each
$\eps \in E_\infty$. To define the scaling functions, recall that we
wrote $ \eps^* = \left( \eps,0,0\dots \right) $ for the canonical
embedding of $E$ into $E_\infty$; the image of this embedding
operation,
\[
E^* = \left\{ \eps^* \;:\; \eps \in E \right\} \subset E_\infty
\]
thus consists of all infinite $0$-$1$--sequences which contain only
a finite number of nonzero components. It is worthwhile to keep in
mind that the subdivision scheme $S_\epsilon$ converges for all
$\eps \in E^*$ if and only if $a_0$ defines a convergent subdivision
scheme and hence the functions
\[
\left\{ f_\eps \;:\; \eps \in E^* \right\} = \left\{ f_{\eps^*}
\;:\; \eps \in
  E \right\}
\]
which will be needed to build the MRA can be ensured to exist by
requiring the existence of an appropriate solution of the refinement
equation associated to $a_0$. This is a much weaker condition, of
course, than convergence of the $S_\eps$ for any $\eps \in
E_\infty$.

\begin{definition}
  The \emph{shearlet scaling spaces} are defined as
  \[
  V_0 = \span \left\{ f_{\eps^*} \left(
      \cdot - \alpha \right) \;:\; \alpha \in \Z^2, \, \eps \in E \right\}
  \]
  and
\[
  V_n = \sum_{\eps \in \{0,1\}^n} V_\eps, \; n \ge 1,
  \]
where
  \[
  V_\eps = \span \left\{ f \left( W_\eps \cdot - \alpha \right)
    \;:\; \alpha \in \Z^2, \, f \in V_0 \right\} \quad \mbox{for all } \eps \in E.
  \]
\end{definition}

Indeed this choice of scaling spaces provides a multiresolution
analysis, which is the focus of the following theorem. The main
ingredient in the proof is -- as it should be -- the refinement
equation \eqref{eq:RefEq}.

\begin{theorem}\label{T:ShearletMRA}
  The spaces $(V_n)_{n \ge 0}$ create a multiresolution analysis. In particular,
  \begin{enumerate}
  \item the spaces $V_n$, $n \ge 0$ are translation invariant,
  \item $V_n \subseteq V_{n+1}$ for all $n \ge 0$, and
  \item for each $n \in \NN$, we have $f \in V_n \Leftrightarrow f(W_\eps \, \cdot) \in V_{n+1}$
  for each $\eps \in \{0,1\}$.
  \end{enumerate}
\end{theorem}

\begin{proof}
  Statement (i) follows immediately from the definition of $V_n$,
  which is a translational completion.

  To verify the nestedness property (ii), we consider an arbitrary
  ``basis element''  $f \in V_n$ of
  the form
  \begin{equation}\label{eq:basiselement}
    f = f_{\eta^*} \left( W_\eps \cdot - \alpha \right), \qquad \eps
    \in E_n, \quad \eta = \left( \eta_1,\widehat \eta \right) \in E, \quad \alpha \in \ZZ^2,
  \end{equation}
  and make use of the refinement equation (\ref{eq:RefEq}) to verify that
  \[
  f = \sum_{\beta \in \Z^2} a_{\eta_1} (\beta) \, f_{{\widehat \eta}^*}
  \left( W_{\eta_1} (W_\eps \cdot - \alpha) - \beta \right)
  = \sum_{\beta \in \Z^2} a_{\eta_1} (\beta - W_{\eta_1} \alpha) \, f_{{\widehat \eta}^*}
  \left( W_{\eps'} \cdot - \beta \right),
  \]
  with $\eps' = \left( \eps,\eta_1 \right) \in E_{n+1}$, hence
  $f \in V_{\eps'} \subseteq V_{n+1}$.

  To verify (iii) we again consider a function element $f \in V_n$
  of the form \eqref{eq:basiselement}. One implication follows from
  \[ f \left( W_\tau \, \cdot \, \right)
  = f_{\eta^*} \left( W_{(\tau,\eps)}\,\cdot - \alpha \right), \qquad
  \tau \in \{0,1\},
  \]
  the other one can be deduced in a similar way by considering
  $f \in V_{n+1}$ and showing that this yields $f \left( W_\tau^{-1} \cdot \,
  \right) \in V_n$ for any $\tau \in \{0,1\}$.
\end{proof}

Notice that for each fixed $\eps \in E$, the set of functions
$f_{\eps^*} \left(\cdot - \alpha \right)$, $\alpha \in \Z^2$, can be
interpreted as being derived from $\delta_\alpha$ by refining with
the subdivision scheme $S_{\eps}$. Since $S_{\eps}$ is
interpolatory, this set of functions is linearly independent.

Some of the scaling functions which generate $V_0$ are plotted in
Figure \ref{fig:delta}. The different orientations due to the
application of the adaptive directional subdivision scheme to the
Dirac delta $\delta_0$ is evident. This fact forces the associated
shearlet spaces to also comprise directionality, hence to react to
directional behavior of the data.

\begin{figure}[h]
\begin{center}
\begin{picture}(300,250)(0,0)
\put(-5,140)%
{\ifpdf
  \includegraphics[width=5cm]{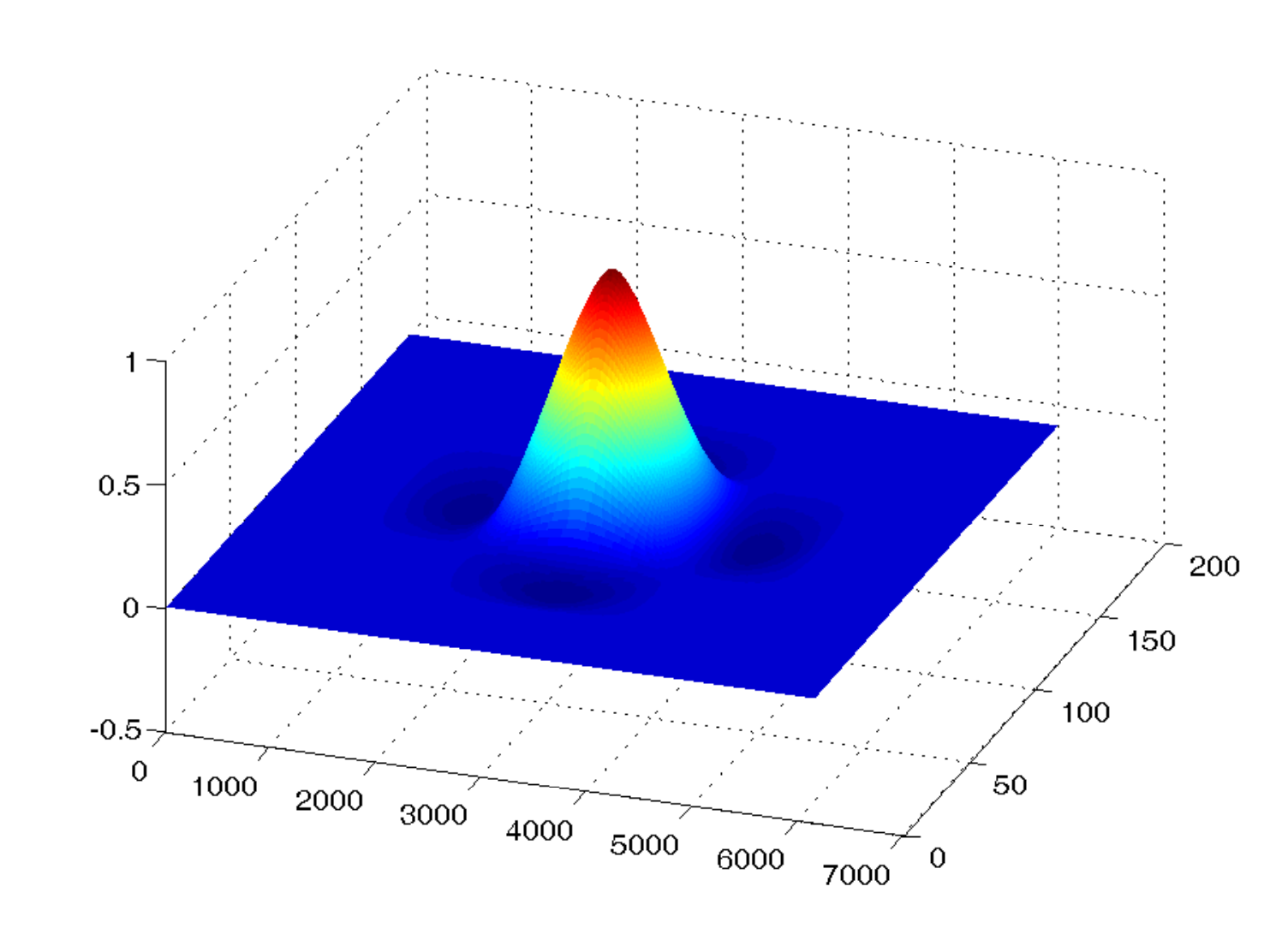}
  \else
  \includegraphics[width=5cm]{delta1.eps}
  \fi}
  \put(60,130){(a)}
  \put(160,140)%
  {\ifpdf
    \includegraphics[width=5cm]{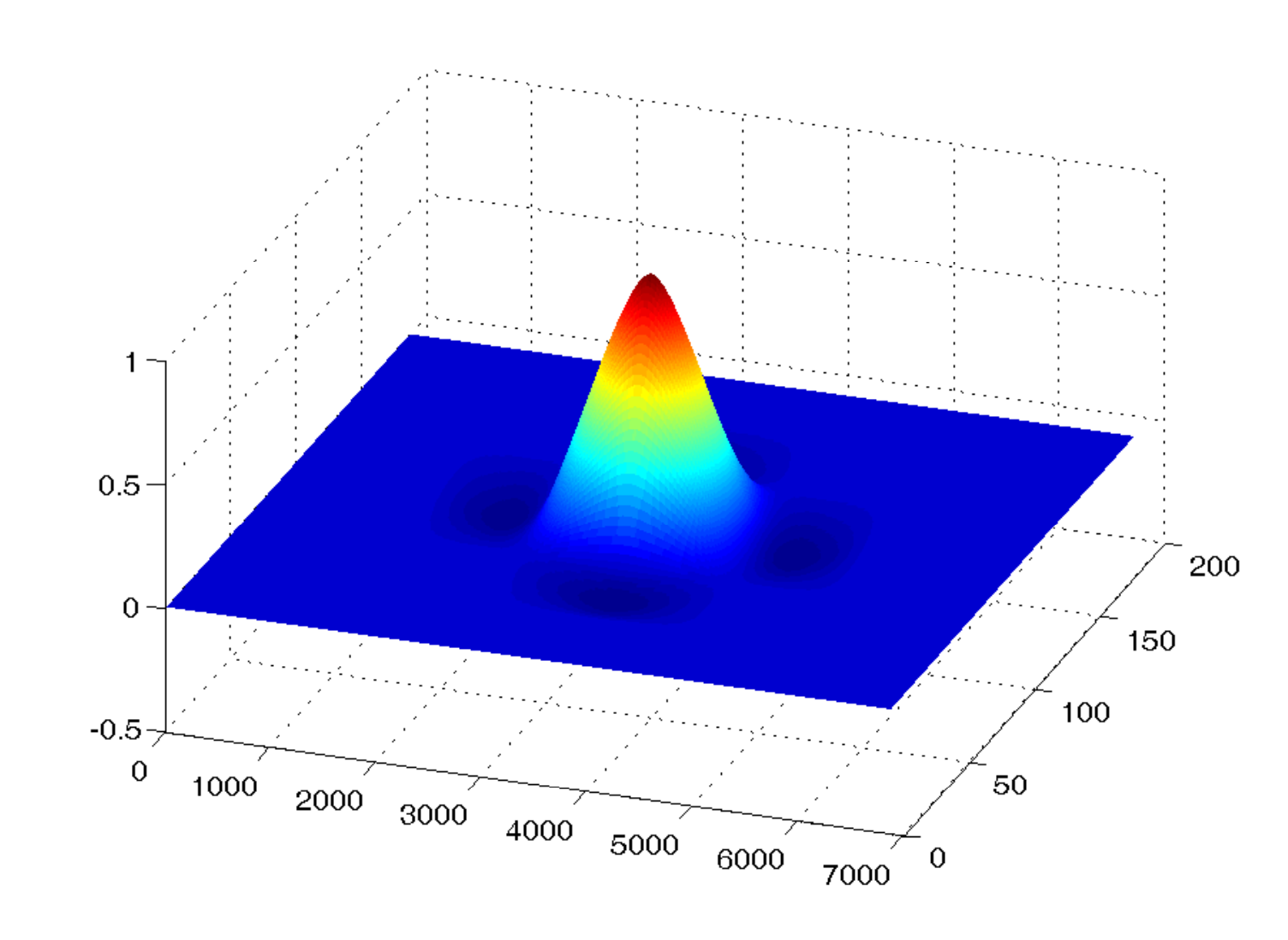}
    \else
    \includegraphics[width=5cm]{delta2.eps}
    \fi}
  \put(225,130){(b)}
  \put(-5,10)%
  {\ifpdf
    \includegraphics[width=5cm]{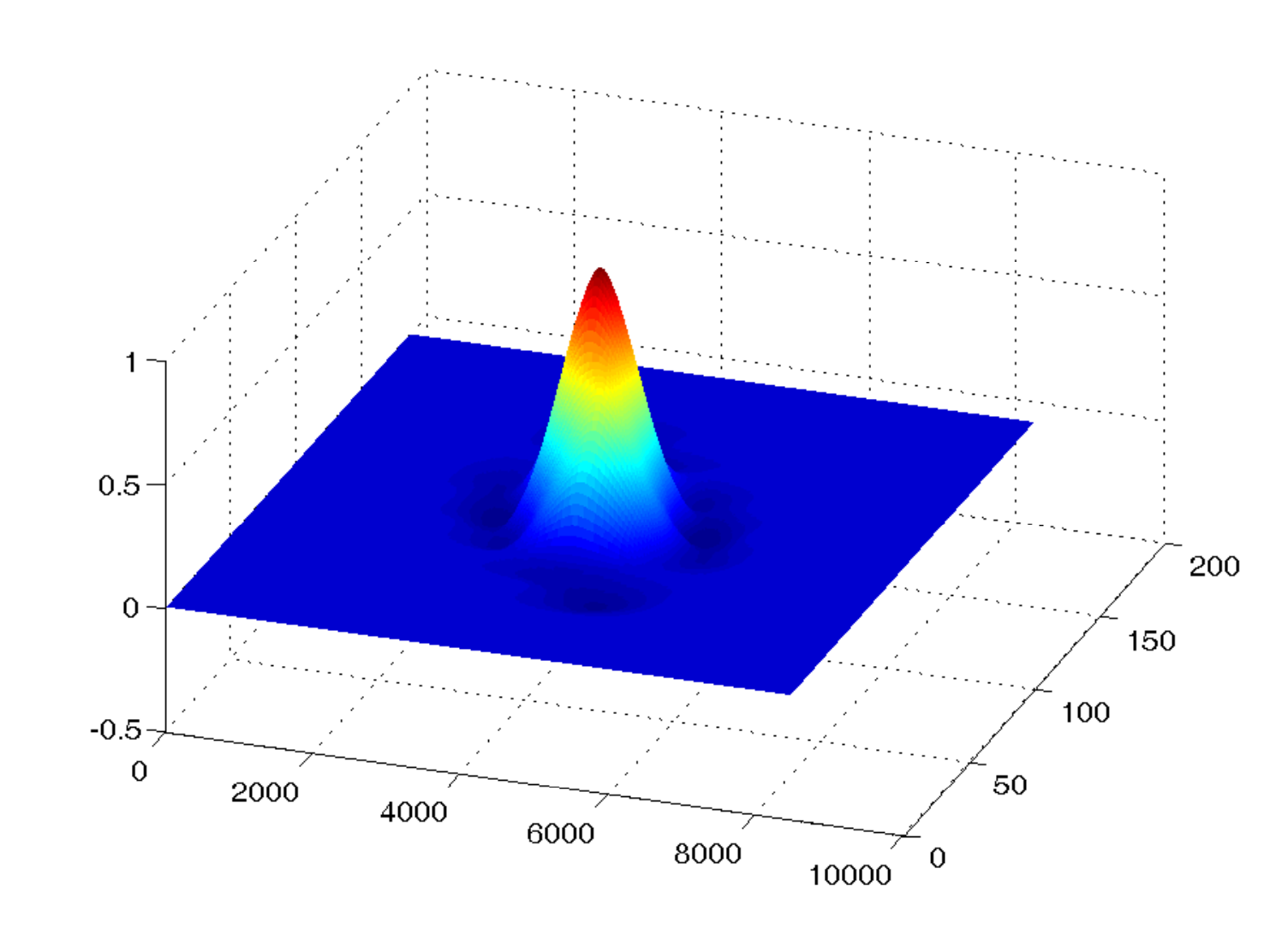}
    \else
    \includegraphics[width=5cm]{delta3.eps}
    \fi}
  \put(60,0){(c)}
  \put(160,10)%
  {\ifpdf
    \includegraphics[width=5cm]{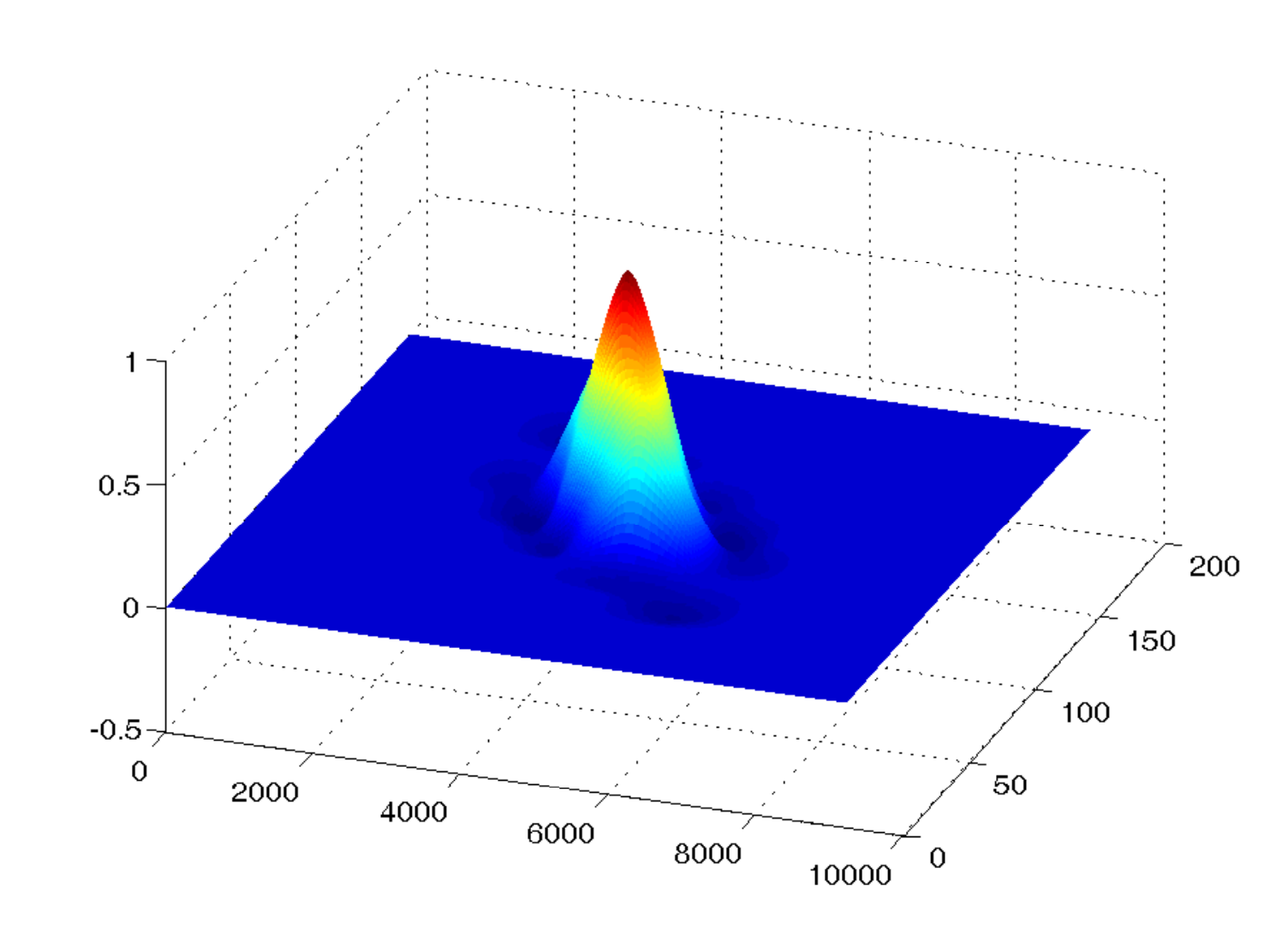}
    \else
    \includegraphics[width=5cm]{delta4.eps}
    \fi}
  \put(225,0){(d)}
  \end{picture}
\end{center}
\caption{This figure shows the refinement of $\delta_0$ after
applying $S_\eps$ with (a) $\eps = (0,0,0,0,0)$,  (b) $\eps =
(0,0,0,1,0)$, (c) $\eps = (0,1,0,0,0)$, and (d) $\eps =
(0,1,1,1,1)$.} \label{fig:delta}
\end{figure}


\section{Fast Shearlet Decomposition}
\label{sec:FSD} Let $\cP_n$, $n \in \NN_0$, denote a sequence of
projections from $V_{n+1}$ to $V_n$, respectively, and define the
\emph{shearlet spaces} as $H_n = \left( \cP_n - I \right) V_{n+1}$,
$n \in \NN_0$, hence as an appropriate complement of $V_n$ in
$V_{n+1}$. In classical MRA, $\cP$ is chosen as an orthogonal
projection, but following the approach from \cite{Faber09}, we can
also use interpolation as a projection, provided that the
subdivision schemes were interpolatory.

\subsection{Refinable Functions}

In order to establish the shearlet decomposition, we require the
following two observations.

\begin{lemma}
  \label{lem:help1}
  For all $\eps \in E$
  and $c \in \ell(\ZZ^2)$, we have
  \[\sum_{\alpha \in \Z^2} c\left( W_\eps^{-1} \alpha
  \right) \, f_0 \left( W_\eps \cdot - \alpha \right)
  = \sum_{\alpha \in \Z^2} c\left( W_0^{-n}  \alpha
  \right) \, f_0 \left(  U_\eps \left( W_0^n \cdot
      - \alpha \right) \right).\]
\end{lemma}

\begin{proof}
  Since all the matrices $U_\eps$, $\eps \in E$, are
  unimodular, we obtain
\begin{eqnarray*}
  \sum_{\alpha \in \Z^2} c\left( W_\eps^{-1} \alpha
  \right) \, f_0 \left( W_\eps \cdot - \alpha \right)
  & = & \sum_{\alpha \in \Z^2} c\left( W_0^{-n} U_\eps^{-1} \alpha
  \right) \, f_0 \left( W_\eps \cdot - \alpha \right)\\
  & = & \sum_{\alpha \in \Z^2} c\left( W_0^{-n}  \alpha
  \right) \, f_0 \left(  U_\eps \left( U_\eps^{-1} W_\eps \cdot
      - \alpha \right) \right)\\
  & = & \sum_{\alpha \in \Z^2} c\left( W_0^{-n}  \alpha
  \right) \, f_0 \left(  U_\eps \left( W_0^n \cdot
      - \alpha \right) \right).\qedhere
\end{eqnarray*}
\end{proof}

\noindent To formulate the next result, we denote by $r : E \to E$
the \emph{reversal} operator for sequences, which maps $\eps =
\left( \eps_1,\dots,\eps_n \right)$ to $r(\eps) := r\left(
\eps_1,\dots,\eps_n \right) := \left( \eps_n,\dots,\eps_1 \right)$.
Moreover, we will write $0_k = P_k 0^*$ for the zero sequence in
$E_k$, $k \in \NN$. We can now derive the following crucial
relationship between refinable functions and subdivision.

\begin{lemma}
  \label{lem:help2}
  For $0 \le k \le n$, $\eps = (\eta,\tau) \in E$, $\eta \in E_k$ and
  $c \in \ell(\ZZ^2)$, we have
  \begin{eqnarray*}
    \sum_{\alpha \in \ZZ^2} c(\alpha) \, f_{\eps^*} \left( W_0^{n-k} \cdot -
      \alpha \right)
    & = & \sum_{\alpha \in \ZZ^2} S_\eta c(\alpha) f_{\widehat{\tau}^*} \left( W_{r(\eta)}
      W_0^{n-k} \cdot - \alpha \right).
  \end{eqnarray*}
\end{lemma}

\begin{proof}
Without loss of generality we can assume that $\tau = (0)$.
  Then, for $\eps = \left( \eps_1,\widehat \eps \right)$,
  the refinement equation \eqref{eq:RefEq} gives
  {\allowdisplaybreaks
  \begin{eqnarray*}
    \lefteqn{
      \sum_{\alpha \in \ZZ^2} c(\alpha) \, f_{\eps^*} \left( W_0^{n-k} \cdot -
        \alpha \right) } \\
    & = & \sum_{\alpha \in \ZZ^2} c(\alpha) \, \sum_{\beta \in \ZZ^2}
    a_{\eps_1} \left( \beta \right)
    f_{\widehat{\eps^*}} \left( W_{\eps_1} \left( W_0^{n-k} \cdot - \alpha
      \right) - \beta \right) \\
    & = & \sum_{\alpha,\beta \in \ZZ^2} a_{\eps_1} \left( \beta - W_{\eps_1}
      \alpha \right) \, c(\alpha) \,
    f_{\widehat{\eps^*}} \left( W_{r \left(\eps_1,0_{n-k} \right)}
      \cdot - \beta \right)
    \\
    & = & \sum_{\beta \in \ZZ^2} \left( S_{\eps_1} c \right)
    \left( \beta \right)
    \, f_{\widehat{\eps^*}} \left( W_{r \left(\eps_1,0_{n-k} \right)}
      \cdot - \beta \right) \\
  \end{eqnarray*}}
  This is the initial step for the inductive proof that for $j \le k$ we have
  \begin{eqnarray}
    \label{eq:lemhelp2pf1}
    \lefteqn{\sum_{\alpha \in \ZZ^2} c(\alpha) \, f_{\eps^*} \left( W_0^{n-k} \cdot -
      \alpha \right)}\\ \nonumber
      & = & \sum_{\beta \in \ZZ^2} S_{\left( \eps_1,\dots,
          \eps_j \right)} c
    \left( \beta \right)
    \, f_{\left( \eps_{j+1},\dots,\eps_k \right)^*}
    \left( W_{r \left(\eps_1,\dots,\eps_j,0_{n-k}\right)} \cdot - \beta
    \right).
  \end{eqnarray}
  Indeed, applying the refinement equation \eqref{eq:RefEq} once more to
  \eqref{eq:lemhelp2pf1}, we get that
  \begin{eqnarray*}
    \lefteqn{
      \sum_{\alpha \in \ZZ^2} c(\alpha) \, f_{\eps^*} \left( W_0^{n-k} \cdot -
        \alpha \right) } \\
    & = & \hspace*{-0.3cm} \sum_{\alpha, \beta \in \ZZ^2} S_{\left( \eps_1,\dots,
        \eps_j \right)} c ( \beta )
    a_{\eps_{j+1}} (\alpha)
    \, f_{\left( \eps_{j+2},\dots,\eps_k \right)^*}
    \left( W_{\eps_{j+1}} \left( W_{r \left(\eps_1,\dots,\eps_j,0_{n-k}
          \right)} \cdot - \beta \right) - \alpha \right) \\
    & = & \hspace*{-0.3cm} \sum_{\alpha, \beta \in \ZZ^2} a_{\eps_{j+1}}
    \left( \alpha - W_{\eps_{j+1}} \beta \right) \, S_{\left( \eps_1,\dots,
        \eps_j \right)} c \left( \beta \right) \, f_{\left( \eps_{j+2},\dots,
        \eps_k \right)^*} \, \left( W_{r \left(\eps_1,\dots,\eps_{j+1},0_{n-k}
        \right)} \cdot - \alpha \right)
  \end{eqnarray*}
  which advances the induction hypothesis in \eqref{eq:lemhelp2pf1}.
  Specifically, for $j = k$ this identity gives
  \begin{eqnarray*}
      \sum_{\alpha \in \ZZ^2} c(\alpha) \, f_{\eps^*} \left( W_0^{n-k} \cdot -
        \alpha \right)
     & = & \sum_{\beta \in \ZZ^2} S_\eps c (\beta) \, f_0 \left(
        W_{r \left(\eps,0_{n-k} \right)} \cdot - \beta \right)  \\
    & = & \sum_{\beta \in \ZZ^2} S_\eps c (\beta) \, f_0 \left( U_{r \left(
          \eps,0_{n-k} \right)}
      W_0^n \cdot - \beta \right)\\
    & = & \sum_{\beta \in \ZZ^2} S_\eps c \left( U_{r \left( \eps,0_{n-k} \right)}
      \beta \right) \, f_0 \left( U_{r \left( \eps, 0_{n-k} \right)}
      \left( W_0^n \cdot - \beta \right) \right).
  \end{eqnarray*}
  Since for any $\eta \in E_k$
  \begin{eqnarray*}
  -2^{n+1} \left[ r \left( \eta,0_{n-k} \right) \right]_2
  & = & -2^{n+1} \, 2^{-n+k} \sum_{j=1}^k \eta_{k-j} 2^{-j}\\
  & = & -2^{k+1} \sum_{j=1}^k r(\eta)_j 2^{-j}\\
  & = & -2^{k+1} \left[ r(\eta) \right]_2,
  \end{eqnarray*}
  we finally get the identity
\begin{eqnarray*}
  \sum_{\alpha \in \ZZ^2} c(\alpha) \, f_{\eps^*} \left( W_0^{n-k} \cdot -
    \alpha \right)
 & = & \sum_{\beta \in \ZZ^2} S_\eps c \left( U_{r \left( \eps \right)}
    \beta \right) \, f_0 \left( U_{r \left( \eps \right)}
    \left( W_0^n \cdot - \beta \right) \right)\\
& = & \sum_{\alpha \in \ZZ^2} S_\eps c(\alpha) f_0 \left( W_{r(\eps)}
      W_0^{n-k} \cdot - \alpha \right),
  \end{eqnarray*}
  which proves the claim.
\end{proof}

Now suppose we are given some data from a finely sampled function on
the grid $W_0^{-n} \ZZ = 4^{-n} \ZZ \times 2^{-n} \ZZ$, say. The key
idea for the decomposition of this data, dependent on different
directions, is stated in the following result which is the backbone
of the MRA based fast discrete shearlet decomposition. We would like
to mention that it relies on the fact that the  masks $a_0$ and
$a_1$ are chosen to be interpolatory and thus give us an explicit
expression for $\cP_n - I$.

The wavelet part of such a decomposition is, as usual, related to
the representatives of the quotient groups $\Gamma_\eps := \ZZ^2 /
W_{r(\eps)} \ZZ^2$, $\eps \in E$. Since for $\eps \in E_n$ we have
$\det W_0 = \det W_1 = 8^n$, all such quotient groups consist of a
number of elements that depends only on the length of $\eps$; we
will denote by $\Gamma_\eps^*$ a selection of $8^n-1$
representatives for $\Gamma_\eps \setminus \{ [0] \}$. In the
sequel, we will make use of the notation $D_M c = c(M\cdot)$, $M$
being some 2$\times$2-matrix.

\begin{theorem}
  \label{theo:decomposition}
  For $c \in \ell(\ZZ^2)$, $\eps = \left( \eta,\tau \right) \in E$, $\eta \in
  E_k$ and $n \ge k$ we have that
  \begin{eqnarray}
    \nonumber
    \lefteqn{ \sum_{\alpha \in \ZZ^2} c( \alpha ) \,
      f_{\tau^*} \left( W_{r(\eta)} W_0^{n-k} \cdot -
        \alpha \right)
      = \sum_{\alpha \in \ZZ^2} c \left( W_{r(\eta)} \alpha \right) \, f_{\eps^*}
      \left( W_0^{n-k} \cdot -
        \alpha \right) } \\
    \label{eq:DecompForm}
    & & + \sum_{\gamma \in \Gamma_{\eta}^*} \sum_{\alpha \in \ZZ^2}
    \left( c - S_\eta D_{W_{r(\eta)}} c \right) \left( W_{r(\eta)} \alpha +
        \gamma \right) \, f_{\tau^*} \left( W_{r(\eta)} \left( W_0^{n-k} \cdot -
        \alpha \right) - \gamma \right).
\end{eqnarray}
\end{theorem}

\begin{proof}
  The decomposition is based on the prediction--correction method which
  has become standard for interpolation based wavelet decomposition, in
  particular in connection with the so--called ``lazy wavelet'' and the
  associated ``lifting schemes'' \cite{Swe96}.

  We subsample the data $c \in \ell \left( \ZZ^2 \right)$ to obtain
  $c' = D_{W_{r(\eta)}}c$ and make use of
  Lemma~\ref{lem:help2} to obtain that
  \[
  \sum_{\alpha \in \ZZ^2} c' (\alpha) \, f_{\eps^*} \left( W_0^{n-k} \cdot -
    \alpha \right)
  = \sum_{\alpha \in \ZZ^2} S_{\eta} c' (\alpha) \, f_{\widehat \tau^*}
  \left( W_{r(\eta)} W_0^{n-k} \cdot - \alpha \right).
  \]
  This identity is then decomposed with respect to $\Gamma_{\eta}$
  giving the \emph{prediction}
  \begin{eqnarray*}
    \lefteqn{
      \sum_{\alpha \in \ZZ^2} c' (\alpha) \, f_{\eps^*} \left( W_0^{n-k}
        \cdot - \alpha \right) } \\
    & = & \sum_{\gamma \in \Gamma_{\eta}} \sum_{\alpha \in \ZZ^2} S_{\eta}
    c' \left( W_{r(\eta)} \alpha + \gamma \right) \, f_{\tau^*} \left(
      W_{r(\eta)} W_0^{n-k} \cdot - W_{r(\eta)} \alpha - \gamma \right) \\
    & = & \sum_{\alpha \in \ZZ^2} S_{\eta} c' \left( W_{r(\eta)} \alpha
    \right) \, f_{\widehat \tau^*} \left( W_{r(\eta)} W_0^{n-k} \cdot -
      W_{r(\eta)} \alpha \right) \\
    & & + \sum_{\gamma \in \Gamma_{\eta}^*} \sum_{\alpha \in \ZZ^2} S_{\eta}
    c' \left( W_{r(\eta)} \alpha + \gamma \right) \, f_{\tau^*} \left(
      W_{r(\eta)} W_0^{n-k} \cdot - W_{r(\eta)} \alpha - \gamma \right) \\
   & = & \sum_{\alpha \in \ZZ^2} c \left( \alpha \right) \, f_{\tau^*} \left(
     W_{r(\eta)} \left( W_0^{n-k} \cdot - \alpha \right) \right) \\
    & & + \sum_{\gamma \in \Gamma_{\eta}^*} \sum_{\alpha \in \ZZ^2} S_{\eta}
    c' \left( W_{r(\eta)} \alpha + \gamma \right) \, f_{\tau^*} \left(
      W_{r(\eta)} \left( W_0^{n-k} \cdot - \alpha \right) - \gamma \right)
  \end{eqnarray*}
  since the subdivision schemes were supposed to be interpolatory. Comparing
  this with the decomposition
  \begin{eqnarray*}
      \sum_{\alpha \in \ZZ^2} c(\alpha) \, f_{\tau^*} \left( W_{r(\eta)}
        W_0^{n-k} \cdot - \alpha \right) \hspace*{-0.05cm}
      & \hspace*{-0.4cm} = \hspace*{-0.4cm}& \hspace*{-0.05cm}\sum_{\gamma \in \Gamma_\eta} \sum_{\alpha \in \ZZ^2} c(\alpha)
      \, f_{\tau^*} \left( W_{r(\eta)}  W_0^{n-k} \cdot - W_{r(\eta)} \alpha
        - \gamma \right)\\
    \hspace*{-0.05cm}& \hspace*{-0.4cm}= \hspace*{-0.4cm}& \hspace*{-0.05cm}\sum_{\alpha \in \ZZ^2} c(\alpha)
    \, f_{\tau^*} \left( W_{r(\eta)}  \left( W_0^{n-k} \cdot - \alpha \right)
    \right)\\
&&    + \sum_{\gamma \in \Gamma_\eta^*} \sum_{\alpha \in \ZZ^2} c(\alpha)
    \, f_{\tau^*} \left( W_{r(\eta)} \left(  W_0^{n-k} \cdot - \alpha \right)
      - \gamma \right)
  \end{eqnarray*}
  we have to apply precisely the \emph{correction} from
  \eqref{eq:DecompForm}.
\end{proof}

\noindent For the special case $\eta = \eps_1$ and thus $\tau =
\widehat \eps$, Theorem~\ref{theo:decomposition} simplifies into the
following form.

\begin{corollary}
  \label{cor:decomposition}
  For $c \in \ell \left( \ZZ^2 \right)$, $\eps \in E$ and $n \in \NN$
  we have that
  \begin{eqnarray}
    \nonumber
    \lefteqn{ \sum_{\alpha \in \ZZ^2} c( \alpha ) \,
      f_{\widehat \eps^*} \left( W_{\eps_1} W_0^{n-1} \cdot - \alpha \right)
      = \sum_{\alpha \in \ZZ^2} c \left( W_{\eps_1} \alpha \right) \,
      f_{\eps^*} \left( W_0^{n-1} \cdot - \alpha \right) } \\
    \label{eq:DecompForm1Step}
    & & + \sum_{\gamma \in \Gamma_{\eps_1}^*} \sum_{\alpha \in \ZZ^2}
    \left( c - S_{\eps_1} D_{W_{\eps_1}} c \right) \left( W_{\eps_1} \alpha +
      \gamma \right) \, f_{\widehat \eps^*} \left( W_{\eps_1} \left( W_0^{n-1} \cdot -
        \alpha \right) - \gamma \right).
  \end{eqnarray}
\end{corollary}

\begin{remark}
  The decomposition \eqref{eq:DecompForm1Step} is the shearlet decomposition
  associated with the shearlet MRA: The function on the left hand side belongs
  to $V_n$ and is written as the sum of a function in $V_{n-1}$ and correction
  terms from $V_n$ that vanish at $W_{\eps_1} \ZZ^2$ -- the \emph{shearlets} in
  the interpolatory MRA.
\end{remark}

\subsection{Decomposition Algorithm}

The \emph{fast shearlet decomposition} is now based on an iterative
application of \eqref{eq:DecompForm1Step}, where each step can be
understood as filtering by means of a filter bank. To that end, we
have to interpret the initial sequence $c \in \ell \left( \ZZ^2
\right)$ appropriately. Denoting by $g_\eps := f_0 \left( U_\eps
\cdot \right)$ the ``sheared'' version of the refinable function
$f_0$, we form the quasi-interpolants
\begin{equation}
  \label{eq:shearMRAquasi}
  q_{\eps,n} :=
  g_\eps * (D_{U_\eps}c) \left( W_0^n \cdot \right) = \sum_{\alpha \in \ZZ^2} c
  \left( U_\eps \alpha \right) \, g_\eps \left( W_0^n \cdot - \alpha \right)
  = \sum_{\alpha \in \ZZ^2} c(\alpha) \, f_0 \left( U_\eps W_0^n \cdot - \alpha
  \right).
\end{equation}
These are precisely the functions which appear on the left hand side
of \eqref{eq:DecompForm} and \eqref{eq:DecompForm1Step}. It is
worthwhile to note that all the functions $q_{\eps,n}$ are relying
on the same initial data $c \in \ell \left( \ZZ^2 \right)$.

The interpretation of \eqref{eq:shearMRAquasi} is rather easy now if
we take into account that $f_0$ was assumed to be the limit function
of an interpolatory scheme, hence cardinal: $f_0 \left( \alpha
\right) = \delta_{0,\alpha}$, $\alpha \in \ZZ^2$. Hence, since
\begin{equation}
  \label{eq:q_nepsinterp}
  q_{\eps,n} (x)
  = \sum_{\alpha \in \ZZ^2} c(\alpha) \, f_0 \left( W_\eps x -
  \alpha \right), \qquad x \in \R^2,
\end{equation}
we can substitute $x = W_\eps^{-1} \alpha = M_\eps \alpha$ and use
the cardinality of $f_0$ to find that $q_{\eps,n} \left( M_\eps
\alpha \right) = c(\alpha)$ or $q_{\eps,n} \left( W_0^{-n} \alpha
\right) = c \left( U_\eps \alpha \right)$, respectively. The latter
tells us that we should interpret the sequence $c$ as a function
sampled at the grid $W_0^{-n} \ZZ^2$, while the parameter $\eps$
determines how this data is sheared and which thus are the
directions ``preferred'' by the wavelet decomposition.

For the fast decomposition we now start with $c \in \ell \left(
\ZZ^2 \right)$, interpret it as in \eqref{eq:q_nepsinterp}, and
decompose it in two ways, namely, for $\eps \in E_1$, into
\[
q_{\eps,n} = \sum_{\alpha \in \ZZ^2} c_\eps \left( \alpha \right) \,
f_{\eps^*} \left( W_0^{n-1} \cdot - \alpha \right) + \sum_{\gamma
\in \Gamma_{\eps}^*} \sum_{\alpha \in \ZZ^2} d_{\eps,\gamma} \left(
\alpha \right) \, f_0 \left( W_{\eps} \left( W_0^{n-1} \cdot -
    \alpha \right) - \gamma \right),
\]
where the coefficients
\begin{eqnarray*}
  c_\eps & = & D_{W_\eps}c \\
  d_{\eps,\gamma} & = & \left( c - S_{\eps} D_{W_{\eps}} c \right) \left( W_{\eps} \cdot +
      \gamma \right)
\end{eqnarray*}
are obtained by filtering the original sequence $c$ in {\em both}
cases. This is the fundamental property of this decomposition
algorithm: even if we decompose \emph{two different} functions,
$q_{\eps,n}$ with $\eps \in E_1$, we have to filter \emph{only one}
data vector to obtain the new set of scaling coefficients $\left\{
c_\eps \;:\; \eps \in E_1 \right\}$ and shearlet coefficients
$\left\{ d_{\eps,\gamma} \;:\; \eps \in E_1, \gamma \in
\Gamma_{\eps}^* \right\}$.

In the next step, the sequences $c_\eps$ and the associated
functions $q_{(\eps,\eta),n-1}$ are decomposed in precisely the same
way, making use of Corollary~\ref{cor:decomposition} again. Like
above, we filter $c_0$ twice to obtain new, further downsampled
sequences $c_{(0,0)}$ and $c_{(0,1)}$ together with the respective
shearlet coefficients $d_{(0,0),\gamma}$, $\gamma \in \Gamma_0^*$
and $d_{(0,1),\gamma}$, $\gamma \in \Gamma_1^*$. In exactly the same
way we obtain $c_{(1,0)}$ and $c_{(1,1)}$ as well as
$d_{(1,0),\gamma}$, $\gamma \in \Gamma_0^*$ and $d_{(1,1),\gamma}$,
$\gamma \in \Gamma_1^*$ by filtering $c_1$. These first two steps of
decomposition are illustrated in Figure~\ref{fig:decomposition}.
\unitlength0.25mm
\begin{figure}[h]
\begin{center}
  \begin{picture}(460,180)(0,0)
\put(220,150){\vector(-2,-1){105}} \put(230,150){\vector(2,-1){105}}
\put(220,160){$c = c_{()^*}$} \put(60,80){$c_{(0)^*} \oplus
d_{(0)^*,\gamma}$} \put(300,80){$c_{(0)} \oplus d_{(1)^*,\gamma}$}
\put(70,70){\vector(-1,-2){25}} \put(80,70){\vector(3,-2){80}}
\put(310,70){\vector(-1,-2){25}} \put(320,70){\vector(3,-2){80}}
\put(0,0){$c_{(0,0)^*} \oplus d_{(0,0)^*,\gamma}$}
\put(120,0){$c_{(0,1)^*} \oplus d_{(0,1)^*,\gamma}$}
\put(240,0){$c_{(1,0)^*} \oplus d_{(1,0)^*,\gamma}$}
\put(360,0){$c_{(1,1)^*} \oplus d_{(1,1)^*,\gamma}$}
  \end{picture}
\end{center}
\caption{The binary tree associated with the fast shearlet
decomposition.} \label{fig:decomposition}
\end{figure}
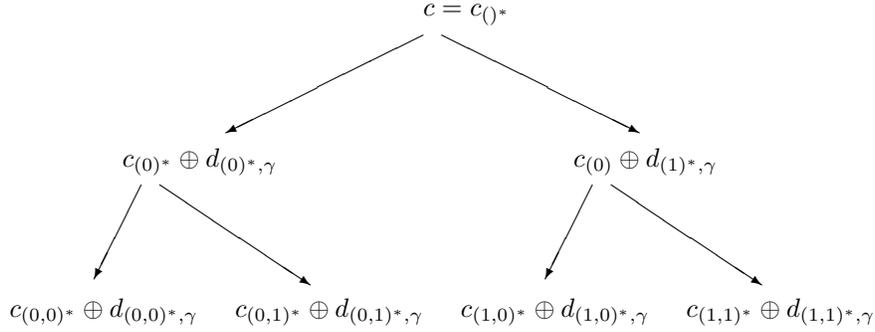
It can already be seen from Figure~\ref{fig:decomposition} that --
like the subdivision scheme -- the shearlet decomposition becomes a
binary tree labeled by the directional indices $\eps$. Indeed, in
general we obtain the new coefficients by the following simple filtering.

\begin{algorithm}
\label{algo:FSD}
Let $c_\eps$ for some $\eps \in E$ be given. Then the next level of scaling and
shearlet coefficients are computed as
\begin{equation*}
  \begin{array}{rcl}
    c_{(\eps,\eta)} & = & D_{W_\eta} c_\eps, \\
    d_{(\eps,\eta),\gamma} & = & \left( c_\eps - S_{\eta} D_{W_{\eta}} c_\eps \right)
    \left( W_{\eta} \cdot + \gamma \right),
  \end{array}
  \qquad \eta \in E_1, \quad \gamma \in \Gamma_{\eta}^*.
\end{equation*}
Eventually, this process ends up with coarsest level scaling
coefficients $c_\eps$, $\eps \in E_n$, and shearlet coefficients
$d_{\eps,\gamma}$, $\eps \in E_k$, $k \le n$, $\gamma \in
\Gamma_{\eps_k}^*$ which describe the deviation from the coarse
data.
\end{algorithm}

Indeed, it is now easily seen that such a decomposition must
recognize ``sheared'' and thus directional components of two
dimensional data since \eqref{eq:DecompForm} relates, for $\eps \in
E$, the data $D_{W_\eps} c$ with the function $g_\eps$ and the
respective shearlet coefficients must be large where the prediction
by the subdivision scheme is inaccurate, i.e., at directional
singularities. Thus, the ``recipe'' is to consider the shearlet
coefficients
\[
d_{P_k \eps, \gamma}, \qquad k = 1,\dots,n, \:\eps \in E_n, \:
\gamma \in \Gamma_{\eps_k}^*.
\]
A precise analysis of this nevertheless fundamental aspect of
directional edge detection is beyond the scope of this paper where
we just want to give the framework for adaptive directional
detections. It should also be clear that the adaptive directional
approach is not tied to interpolatory schemes, in fact, any perfect
reconstruction filter bank can be used as long as the projection and
its complement can be expressed properly. We plan to address these
questions as well as the numerical implementations in a further
paper, however.


\section*{Acknowledgments}

The first author would like to thank Ingrid Daubechies for very
inspiring discussions, and Wolfgang Dahmen for helpful comments on
an earlier version of this paper.
She especially thanks PACM at Princeton University for its
hospitality and support during her visit.

\bibliographystyle{amsplain}

\end{document}